\documentclass[letterpaper, oneside]{amsart}
\usepackage{layout}	

\usepackage[
]{geometry}

\usepackage{amssymb}
\usepackage{amsthm}
\usepackage{amsmath}
\usepackage{enumitem}
\usepackage{mathtools}
\usepackage{hyperref}
\usepackage{colonequals}
\usepackage{mathrsfs}
\usepackage{color}
\usepackage{array}
\usepackage{subcaption}
\usepackage{adjustbox}
\usepackage{ytableau}
\usepackage{bbm}
\usepackage{stmaryrd}

\usepackage{tikz}
\usetikzlibrary{cd}
\usetikzlibrary{decorations.markings}
\usetikzlibrary{arrows.meta}
\usetikzlibrary{patterns}

\tikzset{
    triple/.style args={[#1] in [#2] in [#3]}{
        #1,preaction={preaction={draw,#3},draw,#2}
    },
    dt/.style={
        radius=0.1
    }
}      

\usepackage{tensor}

\theoremstyle{plain}
\newtheorem{thm}{Theorem}[section]
\newtheorem{lem}[thm]{Lemma}

\theoremstyle{definition}

\newtheorem*{ack}{Acknowledgement}

\theoremstyle{remark}
\newtheorem*{rmk}{Remark}

\numberwithin{equation}{section}

\newcommand{\bC}{\mathbb{C}}
\newcommand{\bF}{\mathbb{F}}

\newcommand{\bN}{\mathbb{N}}
\newcommand{\bQ}{\mathbb{Q}}
\newcommand{\bR}{\mathbb{R}}

\newcommand{\bZ}{\mathbb{Z}}

\newcommand{\kk}{\mathbf{k}}

\newcommand{\cB}{\mathcal{B}}
\newcommand{\cH}{\mathcal{H}}

\newcommand{\cO}{\mathcal{O}}

\newcommand{\cN}{\mathcal{N}}
\newcommand{\cK}{\mathcal{K}}

\newcommand{\fg}{\mathfrak{g}}

\newcommand{\fP}{\mathfrak{P}}

\newcommand{\Fq}{\mathbb{F}_q}

\newcommand{\Fqbar}{\overline{\mathbb{F}_q}}
\newcommand{\Ftbar}{\overline{\mathbb{F}_2}}

\newcommand{\und}[1]{\underline{#1}}

\newcommand{\spn}[1]{\textup{span}\left(#1\right)}

\DeclarePairedDelimiter\ceil{\lceil}{\rceil}
\DeclarePairedDelimiter\floor{\lfloor}{\rfloor}
\DeclarePairedDelimiter\bs{[}{]}
\DeclarePairedDelimiter\br{\langle}{\rangle}
\DeclarePairedDelimiter\dbs{\llbracket}{\rrbracket}

\newcommand{\qbinom}[2]{\genfrac{[}{]}{0pt}{}{#1}{#2}}
\newcommand{\dqbinom}[2]{\genfrac{\llbracket}{\rrbracket}{0pt}{}{#1}{#2}}

\DeclareMathOperator{\im}{\textup{im}} 
\DeclareMathOperator{\Ind}{\textup{Ind}}

\DeclareMathOperator{\Ad}{\textup{Ad}}

\DeclareMathOperator{\Irr}{\textup{Irr}}
\DeclareMathOperator{\Id}{\textup{Id}}
\DeclareMathOperator{\Lie}{\textup{Lie}}

\DeclareMathOperator{\sym}{\mathfrak{S}}

\DeclareMathOperator{\Krew}{\textnormal{Krew}}
\DeclareMathOperator{\Nar}{\textnormal{Nar}}
\DeclareMathOperator{\Rx}{\textnormal{Ref}}
\DeclareMathOperator{\NC}{\textnormal{NC}}

\title{$q$-Kreweras numbers for coincidental Coxeter groups attached to limit symbols}
\author{Dongkwan Kim}
\address{School of Mathematics\\
  University of Minnesota Twin Cities\\
  Minneapolis, MN 55455\\
  U.S.A.}
\email{kim00657@umn.edu}
\date{\today}							

\begin{document}

\begin{abstract} For a coincidental Coxeter group, i.e. of type $A_{n-1}, BC_n, H_3,$ or $I_2(m)$, we define the corresponding $q$-Kreweras numbers attached to limit symbols in the sense of Shoji. The construction of these numbers resembles the argument of Reiner and Sommers and these two share similar properties, but our version is slightly different from and combinatorially simpler than theirs. We also study the combinatorial properties of our $q$-Kreweras numbers, i.e. positivity, relation with $q$-Narayana numbers, and cyclic sieving phenomenon.
\end{abstract}

\setcounter{tocdepth}{1}
\setcounter{section}{-1}

\maketitle

\renewcommand\contentsname{}
\tableofcontents

\section{Introduction}
Catalan numbers are one of the most prominent sequences in algebraic combinatorics and have a lot of applications. For example, this sequence counts binary trees with fixed vertices, Dyck paths of fixed length, non-crossing partitions of points lying on the circumference of a circle, triangulation of a regular polygon, to name a few. We refer readers to Stanley's expository monograph \cite{sta15} for more details and its history.

This series has a generalization to any complex reflection group $W$ so that the original one corresponds to the symmetric group. Moreover, similar to many other sequences in algebraic combinatorics, these Catalan numbers admit a $q$-deformation. (cf. \cite{gg12}) Namely, suppose that $d_1, \ldots, d_n$ are the degrees of fundamental invariants of $W$. Then the corresponding $q$-Catalan number with parameter $t$ is given by
\[ \textnormal{Cat}(W, t;q) = \prod_{i=1}^r \frac{\bs{t-1+d_i}}{\bs{d_i}} \]
where $\bs{k}_y=\sum_{i=0}^{k-1} y^i$. If $W$ is the symmetric group permuting $n$ elements, $t=n+1$, and $q=1$ then it becomes the usual Catalan number $C_n=\frac{1}{n+1}\binom{2n}{n}$.

When the complex reflection group $W$ is a coincidental type, i.e. if $\{d_1, \ldots, d_n\}$ forms an arithmetic sequence, the corresponding $q$-Catalan numbers are refined by another $q$-sequence called $q$-Narayana numbers. Following \cite{rss20}, for given $k$ and $t$ we define 
\[\Nar(W, k, t;q) = q^{(t-ak-1)(n-k)} \qbinom{n}{k}_{q^a} \frac{\prod_{i=0}^{k-1}(1-q^{t-1-ai})}{\prod_{i=0}^{k-1}(1-q^{e+1+ai})}.\]
where $a, e \in \bZ_{>0}$ satisfies $\{d_1, \ldots, d_n\} = \{e+ia \mid 0\leq i \leq n-1\}$ and $\qbinom{n}{k}_{q^a}= \prod_{i=n-k+1}^n [i]_{q^a} /\prod_{i=1}^k [i]_{q^a}$. This sequence then satisfies that $\textnormal{Cat}(W, t;q)=\sum_{k=1}^{n}\Nar(W, k, t;q)$. This decomposition also has an interpretation in terms of non-crossing partitions; if $W$ is the symmetric group permuting $n$ elements, $t=n+1$, and $q=1$ then $\Nar(W, k, n+1;1)$ enumerates the number of non-crossing partitions of $n$ with exactly $k$ pieces.

The main object in this paper is $q$-Kreweras numbers, which refines $q$-Narayana numbers when $W$ is of coincidental type. Let $\Irr(W)$ be the set of complex irreducible representations of $W$ up to isomorphism. For each $\chi \in \Irr(W)$, we wish to find a systematic way to define $\Krew(W, \chi, t;q)$, called the $q$-Kreweras number attached to $\chi$, such that $\Nar(W, k, t;q)= \sum_{f(\chi) = k} \Krew(W, \chi, t;q)$ where $f: \Irr(W) \rightarrow \bN=\{0,1,2, \ldots\}$ is a certain level function. Moreover, if $W$ is the symmetric group permuting $n$ elements, $t=n+1$, and $q=1$, so that $\chi$ corresponds to some partition $\lambda \vdash n$, then we hope that $\Krew(W, \chi, n+1;1)$ enumerates the number of non-crossing partitions where the sizes of parts are given by $\lambda$.

In \cite{rs18}, Reiner and Sommers defined a version of $q$-Kreweras numbers for Weyl groups which satisfies nice combinatorial properties. More precisely, they defined $q$-Kreweras numbers indexed by the nilpotent orbits in the corresponding Lie algebra instead of $\Irr(W)$ and show that they are refinements of $q$-Narayana numbers when $W$ is of type $A, B,$ and $C$. Moreover, they proved that their $q$-Kreweras numbers enjoy cyclic sieving phenomena with respect to the generalization of non-crossing partitions studied by Armstrong \cite{arm09}.

Their argument is based on Springer theory. Namely, Sommers \cite{som11} studied the decomposition of a certain virtual graded $W$-representation $\cH_t$, which naturally arises in the representation theory of double affine Hecke algebras, into the Green functions coming from the Springer theory of the corresponding reductive group. Then $q$-Kreweras numbers are defined to be the ``coefficients'' in this decomposition. However, since the Springer correspondence (in type $B, C$, and $D$) from the set of nilpotent orbits (with respect to the trivial local system) to $\Irr(W)$ is not in general bijective but only injective, their $q$-Kreweras numbers are not parametrized by $\Irr(W)$ but nilpotent orbits. Indeed, if one tries to expand their definition to all of $\Irr(W)$ using the Springer correspondence with respect to various local systems of nilpotent orbits, then some of $q$-Kreweras numbers become identically zero.

One of the main goals of this paper is to overcome this limitation. There is a general method to calculate the Green functions, the algorithm originally developed by Shoji \cite{sho83} and Lusztig \cite{lus86}, which is now commonly called the Lusztig-Shoji algorithm. Later it is generalized by Shoji \cite{sho01}, \cite{sho02} so that it is applied to complex reflection groups $G(e, p, n)$. In particular, he investigated the Green functions attached to so-called limit symbols \cite{sho04}, which are combinatorially simpler than the usual Green functions.

Motivated from his work, when $W$ is a coincidental Coxeter group, i.e. of type $A_{n-1}$, $BC_n$, $H_3$, or $I_2(m)$, we define the $q$-Kreweras numbers similarly to the argument of Reiner and Sommers but using the Green functions attached to limit symbols. This definition has certain advantages compared to their work. First, our method allows to define $q$-Kreweras numbers for type $H_3$ and $I_2(m)$, which are not crystallographic so the Springer theory in the usual sense is not applicable. Moreover, in type $BC_n$ the $q$-Kreweras numbers only depend on the type of $W$ whereas in \cite{rs18} there are two different series coming from the Springer theory of $SO_{2n+1}$ and $Sp_{2n}$, respectively. Also, our $q$-Kreweras numbers are indexed by $\Irr(W)$, or more precisely none of $q$-Kreweras numbers is identically zero. We believe that our version is combinatorially simpler and more uniform than their results.


The main results of this paper are summarized as follows. Here $W$ is a coincidental Coxeter group, i.e. of type $A_{n-1}$, $BC_n$, $H_3$, or $I_2(m)$.
\begin{enumerate}[label=\textbullet]
\item For each $\chi \in \Irr(W)$, we provide the closed formula of the $q$-Kreweras number $\Krew(W, \chi, t;q)$ (which also shows that it is not identically zero).
\item Theorem \ref{thm:int}: we show that $\Krew(W, \chi, t;q) \in \bZ[q]$ when $t$ is very good (see Theorem \ref{thm:verygood} for the definition of very good $t$).
\item Theorem \ref{thm:pos}: when $t$ is very good, we show that $\Krew(W, \chi, t;q) \in \bN[q]$ if and only if $\chi$ corresponds to a parabolic subgroup of $W$ (see \ref{sec:pos} for the correspondence $\Phi$ from parabolic subgroups from $W$ to $\Irr(W)$).
\item Theorem \ref{thm:nar}: we show that $\Krew(W, \chi, t;q)$ refines the $q$-Narayana numbers.
\item Theorem \ref{thm:cyc}: we prove that $\Krew(W, \chi, t;q)$ exhibits certain cyclic sieving phenomena with respect to the generalization of non-crossing partitions defined by Armstrong \cite{arm09}.
\end{enumerate}

Our argument is given case-by-case. For type $A_{n-1}$, the Green functions attached to limit symbols are the same as the ones coming from the usual Springer theory, and thus there is nothing new compared to \cite{rs18}. For type $BC_n$, we exploit the Springer theory of $\Lie Sp_{2n}(\Ftbar)$ and the exotic nilcone, which are known to have a strong connection with the Green functions attached to limit symbols. For type $H_3$ and $I_2(m)$, our proof mainly relies on direct calculation. For these results, we use the computer program packages such as \cite{chevie}, \cite{lsalg}, and \cite{sagemath}.

In this paper we only discuss coincidental Coxeter groups. However, the Green functions attached to limit symbols are well-defined for any complex reflection group of the form $G(r,p,n)$. Therefore, it is natural to consider a generalization of our results to such complex reflection groups, or at least for coincidental types. Moreover, even when we restrict our attention to Coxeter groups, we may choose different symbols to perform the Lusztig-Shoji algorithm, which provides different kinds of Green functions. It would be interesting to ask for which choice of symbols such $q$-Kreweras numbers are well-defined and enjoy nice combinatorial properties.

This paper is organized as follows. In Section \ref{sec:defnot} we discuss basic definitions and notations used throughout this paper; in Section \ref{sec:coin} we recall the definition and the classification of coincidental Coxeter groups and parametrizations of their complex irreducible representations; in Section \ref{sec:krelim} we study the definition and the properties of $q$-Kreweras numbers and state the main theorems; from Section \ref{sec:typeA} to the end we prove the main theorems case-by-case.

\begin{ack} The author is grateful to Eric Sommers for his helpful comments and suggestions. In particular, Section \ref{sec:h3} and \ref{sec:i2m} would not be written without his suggestion to look into other coincidental ones than type $A$ and $BC$.
\end{ack}

\section{Definitions and notations}\label{sec:defnot}
For $a, b \in \bZ$, we set $[a,b] = \{ x\in \bZ \mid a\leq x \leq b\}$.
For $x \in \bR$, we set $\floor{x}$ to be the largest integer not greater than $x$ and $\ceil{x}$ to be the smallest integer not smaller than $x$. In other words, we have $\floor{x}, \ceil{x} \in \bZ$, $\floor{x}\leq x < \floor{x}$, and $\ceil{x}-1<x\leq \ceil{x}$.

For $m \in \bN$, we define $\bs{m}_y=\sum_{i=0}^{m-1}y^i$. If $y\neq 1$ then we have $\bs{m}_y = (y^m-1)/(y-1)$. We set $\bs{m}_y! = \prod_{i=1}^m\bs{m}_y$ and 
\[\qbinom{m_1+m_2+\cdots+m_k}{m_1, m_2, \ldots, m_k}_y=\frac{\bs{m_1+m_2+\cdots+m_k}_y!}{\bs{m_1}_y!\bs{m_2}_y!\cdots\bs{m_k}_y!}.\]
 When $k=2$, we also write $\qbinom{m_1+m_2}{m_1}_y=\qbinom{m_1+m_2}{m_2}_y$ instead of $\qbinom{m_1+m_2}{m_1,m_2}_y$. When $y$ equals $q$ or $q^2$, we also write $\bs{m}_q = \bs{m}$, $\bs{m}_{q^2} = \dbs{m}$, $\bs{m}_q!=\bs{m}!$, $\bs{m}_{q^2}!=\dbs{m}!$, etc. Note that $\bs{m}_y, \bs{m}_y!, \qbinom{m_1+m_2+\cdots+m_k}{m_1, m_2, \ldots, m_k}_y$ are elements in $\bN[y]$, i.e. polynomials in $y$ with nonnegative integer coefficients.

A partitions is an integer sequence $\lambda=(\lambda_1, \ldots, \lambda_l)$ such that $\lambda_1\geq \cdots \geq \lambda_l >0$. In such a case we define its length to be $l(\lambda)=l$ and its size to be $|\lambda|=\sum_{i=1}^l\lambda_i$. If $i>l$ then we set $\lambda_i=0$. We set the weighted size of $\lambda$ to be $z(\lambda) =\sum_{i\geq 1}(i-1)\lambda_i = \sum_{j\geq 1}\lambda_j^T(\lambda_j^T-1)/2$ where $\lambda^T = (\lambda_1^T,\lambda_2^T, \ldots)$ is the conjugate partition of $\lambda$.
We write $\fP_n$ to be the set of partitions of size $n$. For $n \in \bZ_{>0}$ and $\lambda=(\lambda_1, \ldots, \lambda_l)$, we set $\lambda_{\geq n}=(\lambda_n, \ldots, \lambda_l)$. For two partitions $\lambda=(\lambda_1, \lambda_2, \ldots)$ and $\mu=(\mu_1, \mu_2, \ldots)$, we define $\lambda+\mu = (\lambda_1+\mu_1, \lambda_2+\mu_2, \ldots)$. Also we define $\lambda \cup \mu$ to be the partition obtained by combining parts of $\lambda$ and $\mu$ and rearranging them so that the result is in decreasing order. Let $m_\lambda(r)$ be the multiplicity of $r$ in $\lambda$. We often write $(1^{m_\lambda(1)}2^{m_\lambda(2)}\cdots)$ to indicate $\lambda$. If $m_\lambda(r) \neq 0$ then we also write $r\in \lambda$.

Let $\fP_{n,2}$ to be the set of pairs of partitions $(\lambda, \mu)$ such that $|\lambda|+|\mu|=n$. Its elements are called bipartitions of $n$. Here we list some notations which will be used in Section \ref{sec:typebc} for the sake of readers' convenience.
\begin{align*}
l(\mu,\nu) &= \max\{l(\mu)-1, l(\nu)\} = l(\mu_{\geq 2}+\nu)
\\m_{\mu,\nu}(r) &= \floor*{\frac{m_{(\mu+\nu)\cup(\mu_{\geq 2}+\nu)}(r)}{2}}
\\L(\mu,\nu) &= l(\mu,\nu) - \sum_{r \geq 1}m_{\mu,\nu}(r)
\\z(\mu,\nu) &= 2z(\mu)+2z(\nu)+|\nu| = 2z(\mu+\nu)+|\nu|
\\d(\mu,\nu)& = \sum_{r\geq 1} m_{\mu,\nu}(r)(m_{\mu,\nu}(r)+1)
\end{align*}

For a finite group $G$, let $\Irr(G)$ be the set of complex irreducible representations of $G$ (up to isomorphism). Also we let $\cK(G)$ be its Grothendieck group which is a free $\bZ$-module with basis $\Irr(G)$, and $\cK(G)^+$ be the submonoid of $\cK(G)$ generated by $\Irr(G)$. For any ring $R$, we write $\cK(G)_R = \cK(G) \otimes R$. 

\section{Coincidental Coxeter groups} \label{sec:coin}
\subsection{Classification of coincidental Coxeter groups}
Let $W$ be a Coxeter group and $e_1, \ldots, e_n$ be its exponents such that $e_1<\cdots <e_n$. We say that $W$ is coincidental if $e_1, \ldots, e_n$ is an arithmetic sequence, i.e. there exists $a\in \bZ_{>0}$ such that $e_i = e_1+(i-1)a$ for $i \in [1,n]$. Such groups are classified as follows.
\begin{enumerate}
\item Type $A_{n-1}$: exponents are $1, 2, \ldots, n-1$ and the Coxeter number is $n$.
\item Type $BC_{n}$: exponents are $1, 3, \ldots, 2n-1$ and the Coxeter number is $2n$.
\item Type $H_3$: exponents are $1, 5, 9$ and the Coxeter number is $10$.
\item Type $I_{2}(m)$: exponents are $1,m-1$ and the Coxeter number is $m$.
\end{enumerate}
Throughout this paper, we assume that $W$ is one of these groups unless otherwise specified.


\subsection{Irreducible representations}
Let us discuss the parametrization of irreducible representations of $W$ (up to isomorphism). Here we adopt Carter's notations, see \cite[Chapter 13.2]{car93}. First suppose that $W$ is of type $A_{n-1}$, so $W$ is the symmetric group permuting $n$ elements. Then its irreducible representations are parametrized by partitions of $n$. Let us write $\chi^\lambda$ to denote such a representation parametrized by $\lambda$. For example, $\chi^{(n)}$ is the identity representation and $\chi^{(1^n)}$ is the sign representation.

This time suppose that $W$ is of type $BC_n$, so $W$ is the hyperoctahedral group of rank $n$. In this case its irreducible representations are parametrized by bipartitions of $n$. Let us write $\chi^{(\mu,\nu)}$ to denote such a representation parametrized by $(\mu,\nu) \in \fP_{n,2}$. For example, $\chi^{((n),\emptyset)}$ is the identity representation and $\chi^{(\emptyset, (1^n))}$ is the sign representation.

Before we proceed let us define the fake degree of an irreducible representation. Let $S^*(V)=\bigoplus_{i\in \bN}S^i(V)q^i$ be the symmetric algebra of the reflection representation of $W$ which is a graded $W$-representation such that each $S^i(V)$ is of degree $i$. (Here, $q$ is a degree-keeping indeterminate.) Then for each $\chi \in \Irr(W)$, there exists an integer $e\in \bN$ such that $\br{\chi, S^i(V)}\neq 0 \Rightarrow e\leq i$ and $\br{\chi, S^e(V)} = 1$. We say that $e$ is the fake degree of $\chi$ and write $b(\chi) = e$.

When $W$ is of type $H_3$, each irreducible representation is completely determined by its dimension and fake degree. We write $\phi_{d,e}$ to be an irreducible representation such that $\dim \phi_{d,e} = d$ and $b(\phi_{d,e}) = e$. Then we have:
\[\Irr(W) = \{\phi_{1,0},\phi_{3,1},\phi_{5,2},\phi_{4,3},\phi_{3,3},\phi_{4,4},\phi_{5,5},\phi_{3,6},\phi_{3,8},\phi_{1,15}\}.\]

Now suppose that $W$ is of type $I_2(m)$. If $m$ is odd then similarly we have
\[\Irr(W) = \{\phi_{1,0}, \phi_{1,m}\} \cup \{\phi_{2,r} \mid r \in [1,(m-1)/2]\}.\]
If $m$ is even, then there are two irreducible representations both of whose dimension is $1$ and fake degree is $m/2$. Here we write $\phi_{1,m/2}', \phi_{1,m/2}''$ to distinguish these two. Then we have
\[\Irr(W) = \{\phi_{1,0},  \phi_{1,m/2}', \phi_{1,m/2}'',\phi_{1,m}\} \cup \{\phi_{2,r} \mid r \in [1,m/2-1]\}.\]

\section{$q$-Kreweras numbers attached to limit symbols} \label{sec:krelim}
We keep assuming that $W$ is a coincidental Coxeter group. Here we provide a definition of $q$-Kreweras numbers attached to limit symbols. To this end we start with the corresponding Green functions of $W$. 
\subsection{Green functions attached to limit symbols}
Recall the Lusztig-Shoji algorithm with respect to limit symbols as follows. We choose a total order on $\Irr(G)$, say $\leq_p$, which satisfies that $\chi \leq_p \chi' \Rightarrow b(\chi) \geq b(\chi')$. In each type this order is chosen as follows.
\begin{enumerate}
\item Type $A_{n-1}$: $\Irr(W)$ is parametrized by $\fP_n$. Here we choose any linearization of dominance order on $\fP_n$.
\item Type $BC_n$: $\Irr(W)$ is parametrized by $\fP_{n,2}$. To each bipartition $(\mu,\nu) \in \fP_{n,2}$ we attach a sequence $(\mu_1, \nu_1, \mu_2, \nu_2, \ldots)$ which is not necessarily a partition. Then we choose any linearization of dominance order on these sequences.
\item Type $H_3$: we choose the order as follows.
\[\phi_{1,0}>_p\phi_{3,1}>_p\phi_{5,2}>_p\phi_{4,3}>_p\phi_{3,3}>_p\phi_{4,4}>_p\phi_{5,5}>_p\phi_{3,6}>_p\phi_{3,8}>_p\phi_{1,15}\]
\item Type $I_2(m)$: if $m$ is odd, then the order is uniquely determined. If $m$ is even then there are two possible choices of linear orders and we take either one.
\end{enumerate}

Now we run the Lusztig-Shoji algorithm following \cite{sho01, sho02, sho04} with respect to the total order described above. If $W$ is of type $A_{n-1}$ or $BC_n$, this process is exactly the same as described in \cite{sho04}. In the case of $H_3$ and $I_2(m)$, there does not exist such a ``limit symbol'' but we may still follow the argument therein to calculate the Green functions with respect to the order described above. By abuse of terminology, we still say that these are the Green functions attached to limit symbols.

\begin{rmk} Since we use a total order to perform the Lusztig-Shoji algorithm, the corresponding ``$P$-matrix'' is not only block triangular but indeed triangular, i.e. each block consists of one irreducible representation. Moreover, the ``$a$-function'' of each irreducible representation is equal to its fake degree.
\end{rmk}

Let $Q_\chi\in \cK(W)_{\bQ(q)}$ be the Green function of $\chi$ as a result of the aforementioned algorithm. In particular, as a graded $W$-representation with $\deg q=1$, $Q_{Id}$ is isomorphic to the coinvariant algebra of $W$. In crystallographic cases, it is the total cohomology of the corresponding flag variety with the usual Springer $W$-action where the degree of the $2i$-th cohomology group is $i$. Then it is known that $\{Q_\chi \mid \chi \in \Irr(W)\}$ is a basis of $\cK(W)_{\bQ(q)}$. Moreover, in our case we have the following stronger statement.
\begin{thm} Suppose that $W$ is a coincidental Coxeter group. Then for $\chi \in \Irr(W)$, $Q_\chi$ (with respect to the limit symbol) is contained in $\cK(W)^+[q]$. In other words, $Q_\chi$ is a genuine $\bN$-graded $W$-representation.
\end{thm}
\begin{proof} For type $A_{n-1}$, $Q_\chi$ is equal to the Green function coming from the usual Springer theory (or Green polynomials), in which case the result is well-known. For type $BC_n$, it follows from \cite[Corollary 5.3]{ah08}. For type $H_3$, it follows from direct calculation. For type $I_2(m)$, it follows from \cite[Theorem 3]{aa08}.
\end{proof}

\subsection{$q$-Kreweras numbers attached to limit symbols}
Here we define the $q$-Kreweras numbers attached to limit symbols similarly to \cite[Section 1.3]{rs18}. To this end, first we introduce a virtual graded $W$-representation 
\[\cH_t = \bigoplus_{i=0}^{n} (-q^t)^i S^*(V) \otimes\wedge^iV \in \cK(W)\llbracket q \rrbracket\]
where $t$ is a positive integer, $n$ is the rank of $W$, $V$ is the reflection representation of $W$, and $S^*(V)=\bigoplus_{i \in \bN} S^i(V)q^i$ is its symmetric algebra as a graded representation of $W$. By Chevalley's result, we have $S^*(V) = Q_{Id}/\prod_{i=1}^n (1-q^{d_i})$ where $\{d_1, \ldots, d_n\}$ are the fundamental degrees of $W$. Thus we have $\cH_t \in \cK(W)_{\bQ(q)}$. Moreover, we have $\cH_t \in \cK(W)^+[q]$ for certain $t$ as the following theorem states.
\begin{thm} \label{thm:verygood} For $t \in \bZ_{>0}$, we say that $t$ is very good if it satisfies the following condition.
\begin{enumerate}
\item Type $A_{n-1}$: $\gcd(t, n)=1$.
\item Type $BC_{n}$: $t$ is odd.
\item Type $H_3$: $t\equiv 1, 5, 9 \pmod{10}$.
\item Type $I_2(m)$: $t\equiv \pm1 \pmod{m}$.
\end{enumerate}
Then, we have $\cH_t \in \cK(W)^+[q]$ if and only if $t$ is very good. In particular, $t$ is very good if $t-1$ is a multiple of the Coxeter number of $W$.
\end{thm}
\begin{proof} The proof of \cite[Proposition 13]{som11} still applies here (even when $W$ is of type $H_3$ or $I_2(m)$).
\end{proof}

Since $\{Q_\chi \mid \chi \in \Irr(W)\}$ is a basis of $\cK(W)_{\bQ(q)}$, $\cH_t$ can be uniquely written as a linear combination of $Q_\chi$. Let us define the $q$-Kreweras numbers, $\Krew(W, \chi, t;q) \in \bQ(q)$, to be such that the following equation holds:
\[\cH_t = \bigoplus_{\chi \in \Irr(W)} \Krew(W, \chi, t;q)Q_\chi.\]
We are especially interested in the case when $t$ is very good. The following theorem will be shown case-by-case in later sections.
\begin{thm}\label{thm:int} For $t$ very good, we have $\Krew(W, \chi, t;q) \in \bZ[q]$ for any $\chi \in \Irr(W)$.
\end{thm}

\subsection{Positivity of $q$-Kreweras numbers} \label{sec:pos}
In general, even if $\Krew(W, \chi, t;q) \in \bZ[q]$, it does not always hold that $\Krew(W, \chi, t;q) \in \bN[q]$. Let us state the sufficient and necessary condition for the positivity of the $q$-Kreweras numbers. To this end, we define a map $\Phi$ from the set of (the conjugacy classes of) parabolic subgroups in $W$ to $\Irr(W)$. 

If $W$ is of type $A_{n-1}$, then a parabolic subgroup is of the form $\sym_{a_1} \times \cdots\times \sym_{a_k}$ for some $a_1, \ldots, a_k \in \bZ_{>0}$ such that $a_1+\cdots+a_k = n$. (Here, $\sym_a$ is the symmetric group permuting $a$ elements.) Let $\lambda$ be the partition of $n$ obtained by rearranging $a_1, \ldots, a_k$ if necessarly. Then we set $\Phi(P) = \chi^\lambda$. Note that this is also equivalent to the ``principal-in-a-Levi'' condition of \cite{rs18}. Indeed, if we let $\cO_\lambda \subset \Lie GL_n(\bC)$ be the nilpotent orbit which contains a regular nilpotent element in the Levi subalgebra corresponding to $P \subset W$, then the (usual) Springer correspondence sends $\cO_\lambda$ to $\chi^\lambda$.

If $W$ is of type $BC_n$, then any parabolic subgroup $P\subset W$ is isomorphic to $H_{b}\times \sym_{a_1} \times \cdots \times \sym_{a_k}$ where $H_{b}$ is a Coxeter group of type $BC_{b}$ and $b+a_1+\cdots+a_k = n$. (Here $b$ can be zero.) In this case, without loss of generality we may assume that $a_1\geq \cdots \geq a_m > b \geq a_{m+1} \geq \cdots \geq a_k$ for some $0\leq m\leq k$. We set
\[ \mu=(b, b, \ldots, b, a_{m+1}, \ldots, a_k)\quad  \textnormal{ and }\quad \nu=(a_1-b, a_2-b, \ldots, a_m-b)\]
where there are $m+1$ $b$'s in $\mu$. Now we set $\Phi(P) = \chi^{(\mu,\nu)}$.

\begin{rmk} This correspondence is different from the usual Springer correspondence of either $SO_{2n+1}(\bC)$ or $Sp_{2n}(\bC)$ that is used in \cite{rs18}. Rather, here we exploit the Springer correspondence of $\Lie Sp_{2n}$ over characteristic 2. This is explained in Section \ref{sec:typebc} in more detail.
\end{rmk}

If $W$ is of type $H_3$, then we are no longer able to argue using the Levi subalgebra of some Lie algebra. Instead, for each $P \subset W$ there exists a unique $\chi\in \Irr(W)$ which appears in $\Ind_P^W \Id_P$ with the highest fake degree. (For such a representation we have $\br{\Ind_P^W \Id_P,\chi}=1$.) Then we set $\Phi(P) = \chi$. We list the types of parabolic subgroups of $W$ and their images under $\Phi$:
\begin{align*}
&\{id\} \mapsto \phi_{1,15}, &&A_1 \mapsto \phi_{3,8}, &&A_1\times A_1 \mapsto \phi_{5,5}, 
\\&A_2 \mapsto \phi_{4,4}, &&I_2(5) \mapsto \phi_{3,3}, &&H_3 \mapsto \phi_{1,0}.
\end{align*}
Note that there are three different parabolic subgroups of type $A_1$; all of them are mapped to the same representation $\phi_{3,8}$.

When $W$ is of type $I_{2}(m)$ we define $\Phi$ similarly to above. If $m$ is odd then we have:
\[\{id\} \mapsto \phi_{1,m}, \quad A_1 \mapsto \phi_{2,(m-1)/2}, \quad I_2(m) \mapsto \phi_{1,0}.\]
If $m$ is even then we have:
\[\{id\} \mapsto \phi_{1,m}, \quad A_1' \mapsto \phi_{1,m/2}', \quad A_1'' \mapsto \phi_{1,m/2}'', \quad I_2(m) \mapsto \phi_{1,0}.\]
Here $A_1'$ and $A_1''$ indicate two different parabolic subgroups, respectively, generated by each simple reflection of $I_2(m)$.

Now we state the positivity theorem of $q$-Kreweras numbers. Its proof is given in later sections case-by-case.
\begin{thm}[cf. {\cite[Theorem 1.6]{rs18}}] \label{thm:pos} For $\chi \in \Irr(W)$ and very good $t$, the corresponding $q$-Kreweras number $\Krew(W, \chi, t;q) \in \bZ[q]$ has nonnegative integer coefficients if and only if $\chi \in \im \Phi$.
\end{thm}

\subsection{Relation with $q$-Narayana numbers}
Recently Reiner-Shepler-Sommers \cite{rss20} defined the $q$-Narayana numbers for any complex coincidental reflection group. Here we recall their definition. Suppose that the exponents of $W$ are $e, e+a, \ldots, e+(n-1)a$ for some $e, a \in \bN$ where $n$ is the rank of $W$. Then the $q$-Narayana number with parameters $k$ and $t$ is defined to be
\[\Nar(W, k, t;q) = q^{(t-ak-1)(n-k)} \qbinom{n}{k}_{q^a} \frac{\prod_{i=0}^{k-1}(1-q^{t-1-ai})}{\prod_{i=0}^{k-1}(1-q^{e+1+ai})}.\]
This definition also coincides with the ones given in \cite{rs18} for type $A_{n-1}$ and $BC_n$.

 In the following sections, we prove the following theorem case-by-case.
\begin{thm}[cf. {\cite[Definition 1.9]{rs18}}] \label{thm:nar}For a coincidental Coxeter group $W$ we have
\[\Nar(W,k,t;q) = \sum_{\br{Q_\chi,V}_{q=1}=k}\Krew(W,\chi, t;q).\]
Here, $\br{Q_\chi,V}_{q=1}$ is the (ungraded) multiplicity of the reflection representation $V$ of $W$ in $Q_\chi$.
\end{thm}
\begin{rmk} In \cite{rs18}, the $q$-Narayana numbers are defined to be the sum of $q$-Kreweras numbers (with respect to the usual Springer theory) similar to the above theorem and the definition in \cite{rss20} is a theorem therein. 
\end{rmk}

\subsection{Cyclic sieving} \label{sec:cycsi}
In \cite{rs18} it was conjectured, and proved for classical types (with respect to the usual Springer theory), that the $q$-Kreweras numbers exhibit certain cyclic sieving phenomena. To this end, first we introduce the notion of (chains of) non-crossing partitions for general Coxeter groups. For a Coxeter system $(W, S=\{s_1, \ldots, s_n\})$, we fix a (standard) Coxeter element, say $c=s_1s_2\cdots s_n$. Let $\Rx(W)$ be the set of all reflections in $W$. (This set is in general strictly larger than $S$.) Define the absolute length of $w \in W$, say $l^a(w)$, to be the minimum number of reflections whose product is $w$. We define an order on $W$ to be the closure of the cover relations $w\leq_a wr$ where $r\in \Rx(W)$ and $l^a(wr)=l^a(w)+1$. Now we define
\[\NC^{(s)}(W) = \{(w_1, \ldots, w_s) \in W \mid w_1 \leq_a \cdots \leq_a w_s \leq_a c\}.\]
This set depends on the choice of $c$ but they are all equivalent under conjugation.

\begin{rmk}
When $W$ is a symmetric group permuting $n$ elements, $\NC^{(1)}(W)$ is equivalent to non-crossing partitions of $n$ elements originally introduced by Kreweras \cite{kre72}. Later it is vastly generalized and the definition above is adopted from \cite{arm09}, to which we refer readers for more details.
\end{rmk}

With $(w_1, \ldots, w_s) \in \NC^{(s)}(W)$ we associate a sequence
\[(\delta_1, \ldots, \delta_{s-1}, \delta_s) = (w_1^{-1}w_2, \ldots, w_{s-1}^{-1}w_s, w_{s}^{-1}c)\]
called a $\delta$-sequence in \cite{arm09}. Note that such a $\delta$-sequence uniquely determines $(w_1, \ldots, w_s)$. Now we define a $\bZ/sh$-action on $\NC^{(s)}(W)$ (where $h$ is the Coxeter number of $W$) so that in terms of a $\delta$-sequence it is described as
\[(\delta_1, \ldots, \delta_{s-1}, \delta_s)  \mapsto (c\delta_s c^{-1}, \delta_1, \ldots, \delta_{s-1}). \]
This action is indeed well-defined and clearly a $\bZ/sh$-action. (See \cite[3.4]{arm09} for detailed discussion.) 

Recall that when $t$ is very good, $\Krew(W, \chi,t;q) \in \bN[q]$ if and only if $\chi \in \im \Phi$. For a parabolic subgroup $P \subset W$, let $V^P$ be the set of point-wise fixed points in the reflection representation $V$ by elements in $P$, and let $W\cdot V^P = \{w\cdot V^P \mid w \in W\} = \{V^{wPw^{-1}}\mid w \in W\}$. The following theorem is shown case-by-case in later sections.
\begin{thm}[cf. {\cite[Conjecure 1.4]{rs18}}] \label{thm:cyc} Let $P$ be a parabolic subgroup of $W$. For $s,d\in \bN$ such that $d \mid sh+1$ where $h$ is the Coxeter number of $W$, the specialization $\Krew(W, \Phi(P), sh+1;\omega_d)$ is equal to the number of elements in 
\[\{(w_1, \ldots, w_s) \in \NC^{(s)}(W) \mid V^{w_1} \in W\cdot V^P \}\]
fixed by the order $d$ element in $\bZ/sh$ with respect to the cyclic action described above. Here, $\omega_d \in \bC^\times$ is a primitive $d$-th root of unity.
\end{thm}

\section{Type $A_{n-1}$}\label{sec:typeA}
In type $A$, the Green functions attached to limit symbols are the same as the usual Green functions (or Green polynomials). Thus, in this case everything is already covered by the results of Reiner and Sommers  \cite{rs18}. Here we review their work for the sake of readers' convenience.

\subsection{$q$-Kreweras numbers and positivity}
Suppose that $W$ is of type $A_{n-1}$. Then for $t \in \bZ_{>0}$ and $\lambda \vdash n$, we have
\[\Krew(W, \chi^\lambda,t;q)=q^{t(n-l(\lambda))-c(\lambda)}\frac{1}{[t]}\qbinom{t}{t-l(\lambda),m_\lambda(\bZ_{>0})}\]
where $\qbinom{t}{t-l(\lambda),m_\lambda(\bZ_{>0})}=\qbinom{t}{t-l(\lambda),m_\lambda(1),m_\lambda(2),\ldots}.$
Here, $c(\lambda) = \sum_{i\geq 1}\lambda^T_i \lambda^T_{i+1}$ where $\lambda^T$ is the conjugate partition of $\lambda$. In particular, if $\gcd(t,n) = 1$ then we have $\Krew(W, \chi^\lambda,t;q) \in \bN[q]$. (Note that in this case $\im \Phi=\Irr(W)$.)

\subsection{Relation with $q$-Narayana numbers}
When $W$ is of type $A_{n-1}$, for $k,t \in \bN$ such that $0\leq k\leq n-1$ and $\gcd(t,n)=1$ the $q$-Narayana number with parameters $k,t$ is given by
\[\Nar(W,k, t;q) = q^{(n-1-k)(t-1-k)}\frac{1}{\bs{k+1}} \qbinom{n-1}{k}\qbinom{t-1}{k}.\]
Then by \cite[Definition 1.9, Theorem 1.10]{rs18} we have
\[\Nar(W,k,t;q) = \sum_{l(\lambda)=k+1}\Krew(W, \chi^\lambda, t;q).\]
Here, $l(\lambda) =k+1$ if and only if $\br{Q_{\chi^\lambda}, V}|_{q=1} = k$.

\subsection{Cyclic sieving}
We recall \cite[Section 6]{rs18}. The Coxeter number of $W$ of type $A_{n-1}$ is $n$, and here we assume that $t=ns+1$ for some $s \in \bN$. Suppose that $P \subset W$ is a parabolic subgroup such that $\Phi(P) = \chi^\lambda$. In this case, for $d \mid sh+1$ we have $\Krew(W, \chi^\lambda, sh+1;\omega_d) \neq 0$ if and only if [$d \mid m_\lambda(r)$ for all $r\in \bZ_{>0}$] or [there exists a unique $r'$ such that $d \nmid m_\lambda(r')$ and it also satisfies $d \mid m_\lambda(r')-1$]. In this case we have
\[\Krew(W, \chi^\lambda, sh+1;\omega_d) =\left\{
\begin{aligned}
&\frac{1}{m}\binom{s}{s-l(\lambda), m_\lambda(\bZ_{>0})} & \textnormal{ if } d=1,
\\&\binom{ns/d}{ns/d-\floor{l(\lambda)/d}, \floor{m_\lambda(\bZ_{>0})/d}} & \textnormal{ otherwise.}
\end{aligned}\right.\]
This is indeed the number of fixed points in $\{(w_1, \ldots, w_s) \in \NC^{(s)}(W) \mid V^{w_1} \in W\cdot V^P \}$ by the order $d$ element in $\bZ/ns$, as expected.

\section{Type $BC_n$}\label{sec:typebc}
Let $W$ be the Weyl group of type $BC_n$. Here we start with the exotic Springer theory first introduced by Kato \cite{kat09}. Later it was revealed by Achar-Henderson \cite{ah08} that it has a strong connection with the Green functions attached to limit symbols.

\subsection{Exotic Springer representations}
Let $\kk$ be an algebraically closed field whose characteristic is not equal to $2$. For instance we may set $\kk=\bC$ or $\kk=\Fqbar$ where $2\nmid q$. Set $G=GL_{2n}(\kk)$ and $\fg=\Lie G$. We regard $G$ as a group of $\kk$-linear automorphism of $\kk^{2n}$ and $\fg$ as the endomorphism algebra of $\kk^{2n}$.

We fix a symplectic form $\br{\ , \ }$ on $\kk^{2n}$. Then there exists an involution $\theta: G\rightarrow G$ such that for any $g\in G$ and $v, w \in \kk^{2n}$ we have $\br{g^{-1}v,w} = \br{v, \theta(g)w}$. This also induces an involution on $\fg$, which we again denote by $\theta: \fg \rightarrow \fg$. We have an eigenspace decomposition $\fg= \fg^+\oplus \fg^-$ where $\fg^\pm = \{ X \in \fg \mid \br{Xv,w} \pm\br{v,Xw}=0 \textup{ for any } v, w \in \kk^{2n}\}$. Note that $G^\theta$ is isomorphic to the symplectic group $Sp_{2n}(\kk)$, and thus its Weyl group is identified with $W$.

Let $\cN^-$ be the set of nilpotent elements in $\fg^-$. The variety $\cN^-\times \kk^{2n}$ is called the exotic nilpotent cone \cite[1.1]{kat09}. There is a diagonal $G^\theta$-action on $\cN^-\times \kk^{2n}$ defined by $g\cdot (X, v) = (\Ad_g(X), g(v))$. Then the $G^\theta$-orbits in $\cN^-\times \kk^{2n}$ are parametrized by $\fP_{n,2}$ as follows. (ref. \cite{ah08}, \cite{nrs18}) For a $G^\theta$-orbit $\cO \subset \cN^-\times \kk^{2n}$, choose any $(N, v) \in \cO$ and let $\lambda$ be the Jordan type of $N$ as an endomorphism on $\kk^{2n}$. Also let $\hat{\lambda}$ be the Jordan types of $N$ on $\kk^{2n}/\kk[N]v=\kk^{2n}/\spn{N^mv \mid m \in \bN}$. Note that $\lambda$ and $\hat{\lambda}$ does not depend on the choice of $(N, v) \in \cO$. Then there exists a unique pair $(\mu,\nu) \in \fP_{n,2}$ so that $\lambda=(\mu+\nu)\cup (\mu+\nu)$ and $\hat{\lambda}=(\mu+\nu)\cup(\mu_{\geq2}+\nu)$. Again, $(\mu, \nu)$ is independent of the choice of $(N, v) \in \cO$ and this gives a bijective correspondence from the set of $G^\theta$-orbits in $\cN^-\times \kk^{2n}$ to $\fP_{n,2}$. From now on we write $\cO_{\mu,\nu}$ to be the orbit parametrized by $(\mu,\nu) \in \fP_{n,2}$.

Let $\cB$ be the flag variety of $G^\theta$ defined to be:
\[\cB=\{F_\bullet = [F_0\subset F_1\subset \cdots \subset F_{2n-1} \subset F_{2n}=\kk^{2n}] \mid \dim F_i = i, \br{F_i,F_{2n-i}}=0 \textup{ for all } i\in [0,2n]\}.\]
For $(N,v) \in \cN^-\times \kk^{2n}$, we define $\cB_{N,v}= \{F_\bullet \in \cB \mid NF_i \subset F_{i} \textup{ for all } i \in [0, 2n] \textup{ and } v \in F_n\}$, called the exotic Springer fiber of $(N, v)$. Then by \cite{kat09} (see also \cite{ss13}), there exists an action of $W$ on $H^i(\cB_{N,v})$ for each $i \in \bZ$, called the exotic Springer representation.

Note that $H^i(\cB_{N,v})$ does not depend on the choice of $(N,v)$ as long as it is contained in a fixed $G^\theta$-orbit of $\cN^-\times \kk^{2n}$. Now we define $Q_{\mu,\nu} = \sum_{i \in \bN} H^{2i}(\cB_{N,v})q^i$
as a graded $W$-representation where $(N,v)$ is any element in $\cO_{\mu,\nu}$. Then the leading term of $Q_{\mu,\nu}$ is given by $\chi^{\mu,\nu}$, i.e. the exotic Springer correspondence is given by $\cO_{\mu,\nu} \mapsto \chi^{\mu,\nu}$. Moreover, $Q_{\mu,\nu}$ coincides with the Green function with respect to $\chi^{\mu,\nu} \in \Irr(W)$ attached to the limit symbols, which means that $Q_{\mu,\nu}=Q_{\chi^{\mu,\nu}}$

\subsection{Springer representation of $\Lie Sp_{2n}(\Ftbar)$} \label{sec:liesp}
There is another Springer theory, i.e. of $\Lie Sp_{2n}$ in characteristic 2, which shares the same Green functions as exotic one. Here we briefly describe its properties.

We start with the parametrization of the nilpotent orbits in $\Lie Sp_{2n}(\Ftbar)$. (ref. \cite{hes79}, \cite[Section 2.6]{xue12}) Let $\Omega$ be the set of pairs $(\lambda, \kappa)$ where $\lambda \in \fP_{2n}$ and $\kappa$ is a function from $\{r \in \lambda\}$ to $\bN$ such that the following conditions hold:
\begin{enumerate}[label=\textbullet]
\item $m_\lambda(r)$ is even if $r\in \lambda$ is odd,
\item for any $r\in \lambda$ we have $0\leq \kappa(r) \leq r/2$,
\item $\kappa(r) = r/2$ if $m_\lambda(r)$ is odd, and 
\item for $r,r' \in \lambda$ such that $r'\leq r$, we have $\kappa(r')\leq \kappa(r)$ and $r' - \kappa(r') \leq r-\kappa(r)$.
\end{enumerate}
For $N \in \Lie Sp_{2n}(\Ftbar)$, we attach $(\lambda, \kappa) \in \Omega$ to its orbit as follows. Let $\lambda\vdash 2n$ be the Jordan type of $N$. For $r\in \lambda$ we define $\kappa$ to be $\kappa(r) = \min\{i \in \bN \mid \br{N^{2i+1}v, v}=0 \textup{ for any } v\in \ker N^r \}.$
Then $(\lambda,\kappa)\in \Omega$ and it is independent of the choice of $N$ in its orbit. This gives a desired parametrization.

We describe the Springer correspondence of $\Lie Sp_{2n}(\Ftbar)$. It is summarized by the bijection $\iota : \Omega\rightarrow \fP_{n,2}$. (This bijection is deduced from, but not exactly the same as, the one defined in \cite[8.1]{xue12:comb}. Also see \cite[Section 13.1]{kim:liesp}.)  For $(\lambda, \kappa) \in \Omega$, choose $s \in \bN$ such that $2s\geq l(\lambda)$. 
We partition $[1, 2s+1]$ into blocks of size 1 or 2 so that:
\begin{enumerate}[label=$-$]
\item $\{i\}$ is a single block if and only if $\kappa(\lambda_i)=\lambda_i/2 $ and
\item other blocks consist of two consecutive integers.
\end{enumerate}
(Here we adopt the convention that $\kappa(0)=0$.)
Note that if $\{i, i+1\}$ is a block then $\lambda_i=\lambda_{i+1}$ (and thus $\kappa(\lambda_i)=\kappa(\lambda_{i+1})$). Now we set $c_i\in \bN$ for $i \in [1, 2s+1]$ to be:
\begin{enumerate}[label=$-$]
\item if $\{i\}$ is a single block then $c_i = \lambda_i/2$, and
\item otherwise if $\{i, i+1\}$ is a block then $c_i = \kappa(\lambda_i)$ and $c_{i+1} =\lambda_i - \kappa(\lambda_i)$.
\end{enumerate}
Now we set $\mu=(c_1, c_3, \ldots, c_{2s+1})$ and $\nu=(c_2, c_4, \ldots, c_{2s})$ (and remove zeroes at the end if necessary so that $\mu$ and $\nu$ do not depend on the choice of $s$). Then, from the definition of $\Omega$ it follows that $(\mu,\nu) \in \fP_{n,2}$. Now we define $\iota(\lambda, \kappa) = (\mu, \nu)$. This is indeed a bijection and describes the Springer correspondence. In particular, if we let $Q_{\lambda, \kappa}$ be the Green function for the nilpotent orbit parametrized by $(\lambda, \kappa) \in \Omega$ then we have $Q_{\lambda, \kappa} = Q_{\iota(\lambda, \kappa)}$.

\subsection{Analogue of Sommers' theorem}
Here we prove an analogue of \cite[Theorem 2]{som11} in our setting. Let $V=\chi^{(n-1),(1)}\in \Irr(W)$ be the reflection representation of $W$.
\begin{thm}[cf. {\cite[Theorem 2]{som11}}] \label{thm:bcref} For $(\mu,\nu) \in \fP_{n,2}$, we have
\begin{gather*}
\sum_{i=0}^n \br{Q_{\mu,\nu}, \wedge^{i} V}y^{i}=y^{l(\mu,\nu)}\prod_{j=1}^{l(\mu,\nu)}(1+yq^{2j-1}),
\\\textnormal{or equivalently } \sum_{i=0}^n \br{Q_{\mu,\nu}, \wedge^{n-i} V}y^{i}=y^{n-l(\mu,\nu)}\prod_{j=1}^{l(\mu,\nu)}(y+q^{2j-1}),
\end{gather*}
where $l(\mu,\nu)=\max \{l(\mu)-1, l(\nu)\} = l(\mu_{\geq 2}+\nu)$.
\end{thm}
The rest of this section is devoted to the proof of this theorem. Here we mainly follow the argument of \cite{som11}. Most of the statements therein remain valid with minor modification. We start with the following lemma.
\begin{lem}\label{lem:bcref} We have $\br{Q_{\mu,\nu}, V}=\sum_{i=1}^{l(\mu,\nu)} q^{2i-1}$.
\end{lem}
\begin{proof} Let us define $Q_{\lambda, \kappa} \in \cK(W)[q]$ as in \ref{sec:liesp}. Then it follows from Spaltenstein's result \cite[Proposition 1.7(b)]{spa91} that $\br{Q_{\lambda, \kappa}, V} = \sum_{i=1}^{\floor{l(\lambda)/2}} q^{2i-1}$. (Note that his result is independent of the characteristic of the base field.) On the other hand, it is easy to see that $\floor{l(\lambda)/2} = l(\iota(\lambda,\kappa))$. Thus the claim follows from the fact that $\iota$ is a bijection.
\end{proof}

Let us define 
\[g_{\mu,\nu}=\sum_{i=0}^n \br{Q_{\mu,\nu}, \wedge^{n-i} V}y^{i} \quad \textnormal{ and } \quad h_{\mu,\nu} = |\cO_{\mu,\nu}(\bF_q)| g_{\mu,\nu},\]
where $|\cO_{\mu,\nu}(\bF_q)|\in \bQ[q]$ is the number of $\bF_q$-points in $\cO_{\mu,\nu}$ when the base field is $\overline{\bF_q}$ (and $\cO_{\mu,\nu}$ is split over $\bF_q$). Also for $\chi\in \Irr(W)$, we define
\[\tilde{\tau}(\chi) = \sum_{i,j\in \bN} \br{S^i(V)\otimes \wedge^jV, \chi}q^iy^j= \sum_{j\in \bN} \br{S^*(V)\otimes \wedge^jV, \chi}y^j\]
where we set $S^*(V) = \sum_{i \in \bN}S^i(V)q^i$. Now we prove the following lemma.
\begin{lem}[{cf. \cite[(9)]{som11}}] \label{lem:som9}For $\chi \in \Irr(W)$, we have
\[(-1)^nq^{n^2} \tilde{\tau}(\chi) \prod_{i=1}^n(q^{2i}-1)=\sum_{(\mu,\nu) \in \fP_{n,2}}h_{\mu,\nu}\br{Q_{\mu,\nu}, \chi}.\]
\end{lem}
\begin{proof}We start with the orthogonality formula of Green functions presented in \cite{sho82}, which is still valid in our situation:
\[q^{n^2}Q_{Id}\otimes \varphi\otimes \chi^{\emptyset, (1^n)} = \sum_{(\mu,\nu) \in \fP_{n,2}}|\cO_{\mu,\nu}(\bF_q)|\br{Q_{\mu,\nu}, \varphi}Q_{\mu,\nu}.\]
Here we use the fact that the number of positive roots in the root system of $W$ is $n^2$ and the centralizer of any $(N, v) \in \cN^- \times \kk^{2n}$ in $G^\theta$ is connected. If we set $ \varphi=\wedge^{n-j} V$ so that $ \varphi \otimes  \chi^{\emptyset, (1^n)} \simeq \wedge^{n-j} V \otimes \wedge^{n} V \simeq \wedge^j V$, then it follows that
\[q^{n^2}Q_{Id}\otimes \wedge^{j}V = \sum_{(\mu,\nu) \in \fP_{n,2}}|\cO_{\mu,\nu}(\bF_q)|\br{Q_{\mu,\nu},\wedge^{n-j}V}Q_{\mu,\nu}.\]
Note that $Q_{Id} = (-1)^nS^*(V)\cdot \prod_{i=1}^n(q^{2i}-1) $ by the formula of Chevalley. Thus we have
\[\left((-1)^nq^{n^2} \prod_{i=1}^n(q^{2i}-1) \right) S^*(V)\otimes \wedge^{j}V = \sum_{(\mu,\nu) \in \fP_{n,2}}|\cO_{\mu,\nu}(\bF_q)|\br{Q_{\mu,\nu},\wedge^{n-j}V}Q_{\mu,\nu}.\]
Now the result follows from taking $\br{-, \chi}$, multiplying $y^j$, and summing up over $j \in \bN$.
\end{proof}

From \cite[Proposition 3.3]{gns99}, we have
\[\tilde{\tau}(\chi^{\mu,\nu})=q^{2z(\mu)+2z(\nu) +|\nu|} \prod_{x' \in \mu} \frac{1+yq^{2c(x')+1}}{1-q^{2h(x')}}\prod_{x'' \in \nu} \frac{1+yq^{2c(x'')-1}}{1-q^{2h(x'')}}\]
where $z(\lambda) = \sum_{i\in \bZ_{>0}}(i-1)\lambda_i$ and $c(b)$ (resp. $h(b)$) is the content($=$(column index)$-$(row index)) (resp. the hook length) of $b$ in the Young diagram. In particular, by considering boxes at the first column of the Young diagrams of $\mu$ and $\nu$, we see that $(1+yq^{-2i+1})\mid\tilde{\tau}(\chi^{\mu,\nu})$ when $i \in [1,l(\mu,\nu)]$. (This statement is vacuous when $l(\mu, \nu) = 0$, i.e. $\mu=(n)$ and $\nu=\emptyset$.)

Before the proof of Theorem \ref{thm:bcref} let us observe the triangularity of Springer representations as follows. For $(\rho, \sigma), (\mu,\nu) \in \fP_{n,2}$, we say that $(\rho, \sigma) \leq (\mu,\nu)$ if
\begin{gather*}
\rho_1+\sigma_1+\cdots+\rho_k+\sigma_k \leq \mu_1+\nu_1+\cdots+\mu_k+\nu_k, \textnormal{ and}
\\\rho_1+\sigma_1+\cdots+\rho_k+\sigma_k+\rho_{k+1} \leq \mu_1+\nu_1+\cdots+\mu_k+\nu_k+\mu_{k+1}
\end{gather*}
for any $k \in \bN$. Then the result of \cite[Theorem 6.3]{ah08} states that $\cO_{\rho, \sigma} \subset \overline{\cO_{\mu,\nu}}$ if and only if $(\rho, \sigma) \leq (\mu,\nu)$. Moreover, by \cite{ss14} and \cite{kat17}, it follows that $\br{Q_{\mu,\nu}, \chi^{\rho, \sigma}}\neq 0$ if and only if $(\rho, \sigma) \leq (\mu,\nu)$. Note that $(\rho,\sigma)\leq (\mu,\nu)$ then $l(\rho, \sigma) \leq l(\mu,\nu)$. In particular, if we set $\chi^{\rho, \sigma} = \wedge^j V$, i.e. $(\rho, \sigma) = ((n-j), (1, 1, \ldots, 1))$, then the above inequalities are translated to $n-j \leq \mu_1$ and $n-j+k \leq \mu_1+\nu_1+\cdots+\mu_k+\nu_k$ for any $1\leq k \leq j$. Thus when $k=j$ we have that $l(\mu,\nu)\leq l(\mu), l(\nu) \leq j$.

We are ready to prove Theorem \ref{thm:bcref}, i.e. $g_{\mu,\nu}=y^{n-l(\mu,\nu)}\prod_{j=1}^{l(\mu,\nu)}(y+q^{2j-1})$. To this end, first we prove that the RHS divides $h_{\mu,\nu}$ by induction on $\dim \cO_{\mu,\nu}$. First, if $\dim \cO_{\mu,\nu}=0$, i.e. $\cO_{\mu,\nu}=\{0\}$ then it follows from the result of Solomon \cite{sol63}. In general, by Lemma \ref{lem:som9} and the triangularity of Springer representations we have
\[(-1)^nq^{n^2} \tilde{\tau}(\chi^{\mu,\nu})\prod_{i=1}^n(q^{2i}-1) = h_{\mu,\nu}q^{ 2z(\mu)+2(\nu)+ |\nu|} +\sum_{(\rho,\sigma)>(\mu,\nu) }h_{\rho,\sigma}\br{Q_{\rho,\sigma}, \chi^{\mu,\nu}}.\]
Here we use the fact that $\deg_q Q_{\mu,\nu} = 2z(\mu)+2z(\nu)+ |\nu| = n^2-\dim \cO_{\mu,\nu}/2$ (see \cite[Remark 2.6]{ss13}). Now by induction hypothesis and the argument above, it follows that $\prod_{j=1}^{l(\mu,\nu)}(y+q^{2j-1}) \mid h_{\mu,\nu}$.

Since $|\cO_{\mu,\nu}(\bF_q)|$ does not have any common factor with $\prod_{j=1}^{l(\mu,\nu)}(y+q^{2j-1})$, it follows that $\prod_{j=1}^{l(\mu,\nu)}(y+q^{2j-1}) | g_{\mu,\nu}$. On the other hand, we previously observed that if $n-j>l(\mu,\nu)$ then $\br{Q_{\mu,\nu}, \wedge^{n-j}V}=0$. Thus it follows that $y^{n-l(\mu,\nu)} \mid g_{\mu,\nu}$ as well. Now by considering the degree with respect to $y$, it follows that $g_{\mu,\nu} = c y^{n-l(\mu,\nu)}\prod_{j=1}^{l(\mu,\nu)}(y+q^{2j-1})$ for some $c \in \kk$. However, $c=1$ since the coefficient of $y^n$ in $g_{\mu,\nu}$ equals 1 (as the trivial representation only occurs at the zeroth cohomology group of the corresponding Springer fiber). This suffices for the proof of Theorem \ref{thm:bcref}.

\subsection{$q$-Kreweras numbers and positivity} \label{sec:bckre}
In this section we find a closed formula of $\Krew(W, \chi^{\mu,\nu},t;q)$. To this end, first we let $m_{\mu,\nu}(r)=\floor{m_{(\mu+\nu)\cup (\mu_{\geq 2}+\nu)}(r)/2}$ for $(\mu,\nu) \in \fP_{n,2}$ and $r \in \bZ_{>0}$. This can also be interpreted as follows. We set $\Lambda=(\mu_1, \nu_1, \mu_2, \nu_2, \ldots)=(a_1, a_2, a_3, a_4, \ldots)$. Then $\Lambda$ is not necessarily a partition, but it is a quasi-partition in the sense of \cite{ahs11}; we have $a_i\geq a_{i+2}$ for $ i \in \bZ_{>0}$. Also we have $\mu+\nu=(a_{2i-1}+a_{2i})_{i \geq 1}$ and $\mu_{\geq 2}+\nu=(a_{2i}+a_{2i+1})_{i \geq 1}$.

\begin{lem} \label{lem:hm} Set $\Lambda=(a_1, a_2, \ldots)$ as above. For $r\in (\mu+\nu)\cup (\mu_{\geq 2}+\nu)$, suppose that $i$ (resp. $j$) is the smallest (resp. largest) index so that $a_{i}+a_{i+1} = r$ (resp. $a_{j}+a_{j+1} = r$). Then, there are three possible cases:
\begin{enumerate}
\item If $i$ and $j$ are both even then $m_{\mu+\nu}(r)=m_{\mu_{\geq 2}+\nu}(r)-1$. Thus we have $m_{\mu,\nu}(r) = m_{\mu+\nu}(r) = m_{\mu_{\geq 2}+\nu}(r)-1$.
\item If $i \not\equiv j \pmod 2$ then $m_{\mu+\nu}(r)=m_{\mu_{\geq 2}+\nu}(r)$. Thus we have $m_{\mu,\nu}(r)= m_{\mu+\nu}(r) =m_{\mu_{\geq 2}+\nu}(r)$.
\item If $i$ and $j$ are both odd then $m_{\mu+\nu}(r)=m_{\mu_{\geq 2}+\nu}(r)+1$. Thus we have $m_{\mu,\nu}(r) = m_{\mu+\nu}(r) -1=m_{\mu_{\geq 2}+\nu}(r)$.
\end{enumerate}
\end{lem}
\begin{proof} It can be shown case-by-case.
\end{proof}
Let $L(\mu,\nu)$ be the number of $r\in (\mu+\nu)\cup (\mu_{\geq 2}+\nu)$ such that the first condition in Lemma \ref{lem:hm} is satisfied. (In this case we necessarily have $r\in \mu_{\geq 2}+\nu$.) In other words, we have $L(\mu,\nu)+\sum_{r>0}m_{\mu,\nu}(r) = l(\mu,\nu)=\sum_{r>0}m_{\mu_{\geq2}+\nu}(r)$. Also we define 
\[z(\mu,\nu) = 2z(\mu)+2z(\nu)+|\nu| = 2z(\mu+\nu)+|\nu|\]
where $z(\lambda) = \sum_{i>0}(i-1)\lambda_i$. We set $d(\mu,\nu)= \sum_{r >0} m_{\mu,\nu}(r)(m_{\mu,\nu}(r)+1)$. 
\begin{thm} \label{thm:bckre} For $(\mu,\nu)\in \fP_{n,2}$, the $q$-Kreweras number $\Krew(W,\chi^{\mu,\nu},t;q)$ is given by 
\begin{align*}
q^{t(n-l(\mu,\nu))+l(\mu,\nu)^2-2z(\mu,\nu) +d(\mu,\nu)-n}\prod_{j=1}^{L(\mu,\nu)}(q^{t-2j+1}-1)
\dqbinom{(t-1)/2-L(\mu,\nu)}{(t-1)/2-l(\mu,\nu),m_{\mu,\nu}(\bZ_{>0})}.
\end{align*}
In particular, when $t$ is odd, $\Krew(W,\chi^{\mu,\nu},t;q) \in \bN[q]$ if and only if $L(\mu,\nu)=0$.
\end{thm}

The rest of this section is devoted to its proof. Using $S^*(V) = Q_{Id}\cdot\prod_{i=1}^n (1-q^{2i})^{-1}$, we have
\begin{align*}
\cH_t&=\sum_{j \in \bN}(-q^t)^j S^*(V) \otimes \wedge^jV =\prod_{i=1}^n (1-q^{2i})^{-1}\sum_{j \in \bN}(-q^t)^j Q_{Id} \otimes \wedge^jV.
\end{align*}
Recall that $q^{n^2}Q_{Id}\otimes \wedge^{j}V = \sum_{(\mu,\nu) \in \fP_{n,2}}|\cO_{\mu,\nu}(\bF_q)|\br{Q_{\mu,\nu},\wedge^{n-j}V}Q_{\mu,\nu}$. Thus we have
\begin{align*}
\cH_t&=q^{-n^2}\prod_{i=1}^n (1-q^{2i})^{-1}\sum_{j \in \bN}(-q^t)^j  \sum_{(\mu,\nu) \in \fP_{n,2}}|\cO_{\mu,\nu}(\bF_q)|\br{Q_{\mu,\nu},\wedge^{n-j}V}Q_{\mu,\nu}.
\end{align*}
It follows that
\[\Krew(W, \chi^{\mu,\nu},t;q)= q^{-n^2}\prod_{i=1}^n (1-q^{2i})^{-1}\sum_{j \in \bN}(-q^t)^j |\cO_{\mu,\nu}(\bF_q)|\br{Q_{\mu,\nu},\wedge^{n-j}V}.\]
On the other hand, if we substitute $y$ with $-q^t$ in the equation of Theorem \ref{thm:bcref} then we have \begin{align*}
\sum_{i=0}^n \br{Q_{\mu,\nu}, \wedge^{n-j} V}(-q^t)^{j} &=(-q^t)^{n-l(\mu,\nu)}\prod_{j=1}^{l(\mu,\nu)}(-q^t+q^{2j-1})
\\&=(-q^t)^{n-l(\mu,\nu)}\prod_{j=1}^{l(\mu,\nu)}(-q^{2j-1})\prod_{j=1}^{l(\mu,\nu)}(q^{t-2j+1}-1)
\\&=(-1)^n q^{t(n-l(\mu,\nu)) +l(\mu,\nu)^2}\prod_{j=1}^{l(\mu,\nu)}(q^{t-2j+1}-1).
\end{align*}
Thus we have
\begin{align*}
\Krew(W, \chi^{\mu,\nu},t;q)&= q^{-n^2}\prod_{i=1}^n (1-q^{2i})^{-1} |\cO_{\mu,\nu}(\bF_q)| \cdot (-1)^n q^{t(n-l(\mu,\nu)) +l(\mu,\nu)^2}\prod_{j=1}^{l(\mu,\nu)}(q^{t-2j+1}-1)
\\&= q^{t(n-l(\mu,\nu)) +l(\mu,\nu)^2-n^2}\prod_{i=1}^n (q^{2i}-1)^{-1}  \prod_{j=1}^{l(\mu,\nu)}(q^{t-2j+1}-1)|\cO_{\mu,\nu}(\bF_q)|.
\end{align*}
We recall the result of Sun \cite[Corollary 3.13]{sun11} that gives the closed formula of $|\cO_{\mu,\nu}(\bF_q)|$:
\[|\cO_{\mu,\nu}(\bF_q)| = q^{2n^2-2z(\mu,\nu)}\frac{ \prod_{i=1}^n(1-q^{-2i})}{\prod_{r\in J} \prod_{i=1}^{m_{\mu+\nu}(r)-1}(1-q^{-2i})\prod_{r\not\in J}\prod_{i=1}^{m_{\mu+\nu}(r)}(1-q^{-2i})}\]
where $r \in J$ if and only if $r$ is in the third case of Lemma \ref{lem:hm}. (See \cite[Notation 2.5]{sun11}.) Therefore, by Lemma \ref{lem:hm} we see that
\begin{align*}
|\cO_{\mu,\nu}(\bF_q)| &= q^{2n^2-2z(\mu,\nu)}\frac{ \prod_{i=1}^n(1-q^{-2i})}{\prod_{r>0}\prod_{i=1}^{m_{\mu,\nu}(r)}(1-q^{-2i})}
\\&=q^{n^2-2z(\mu,\nu)+d(\mu,\nu)-n}\frac{ \prod_{i=1}^n(q^{2i}-1)}{\prod_{r>0}\prod_{i=1}^{m_{\mu,\nu}(r)}(q^{2i}-1)}.
\end{align*}
We substitute $|\cO_{\mu,\nu}(\bF_q)|$ with this formula in the expansion of $\Krew(W,\chi^{\mu,\nu},t;q)$ to see that
\begin{align*}
&\Krew(\mu,\nu;t)(q)= q^{t(n-l(\mu,\nu)) +l(\mu,\nu)^2-2z(\mu,\nu)+d(\mu,\nu)-n} \frac{\prod_{j=1}^{l(\mu,\nu)}(q^{t-2j+1}-1)}{\prod_{r>0}\prod_{i=1}^{m_{\mu,\nu}(r)}(q^{2i}-1)}
\\&= q^{t(n-l(\mu,\nu)) +l(\mu,\nu)^2-2z(\mu,\nu)+d(\mu,\nu)-n}\prod_{j=1}^{L(\mu,\nu)}(q^{t-2j+1}-1)
\frac{ \prod_{j=1}^{l(\mu,\nu)-L(\mu,\nu)}(q^{t-2j-2L(\mu,\nu)+1}-1)}{\prod_{r>0}\prod_{i=1}^{m_{\mu,\nu}(r)}(q^{2i}-1)}
\\&= q^{t(n-l(\mu,\nu)) +l(\mu,\nu)^2-2z(\mu,\nu)+d(\mu,\nu)-n}
\prod_{j=1}^{L(\mu,\nu)}(q^{t-2j+1}-1)
\dqbinom{(t-1)/2-L(\mu,\nu)}{(t-1)/2-l(\mu,\nu),m_{\mu,\nu}(\bZ_{>0})},
\end{align*}
which is what we want to prove.

\subsection{Relation with $q$-Narayana numbers}
When $W$ is of type $BC_n$, due to Lemma \ref{lem:bcref}, Theorem \ref{thm:nar} reads
\[\Nar(W,k,t;q) =\sum_{l(\mu,\nu)=k}\Krew(W,\chi^{\mu,\nu}, t;q).\]
First, note that we have
\[\Nar(W,k,t;q)=q^{(n-k)(t-1-2k)}\dqbinom{n}{k}\dqbinom{(t-1)/2}{k}.\]
On the other hand, from \ref{sec:bckre} we have
\begin{align*}
\Krew(W, \chi^{\mu,\nu},t;q)&= q^{t(n-l(\mu,\nu)) +l(\mu,\nu)^2-n^2}\prod_{i=1}^n (q^{2i}-1)^{-1}  \prod_{j=1}^{l(\mu,\nu)}(q^{t-2j+1}-1) |\cO_{\mu,\nu}(\bF_q)|.
\end{align*}
Comparing the formulas above, it suffices to show that
\begin{align*}
\sum_{ l(\mu,\nu)=k}|\cO_{\mu,\nu}(\bF_q)| 
=q^{(n-k)(n-k-1)}(q^2-1)^{n-k}\frac{\dbs{n}!^2}{\dbs{n-k}!\dbs{k}!^2}.
\end{align*}

Note that the RHS is the same as the one given in \cite[Lemma 5.1]{rs18}. Indeed, let $\fP_{2n}^C\subset \fP_{2n}$ be the set of Jordan types of nilpotent elements in $\Lie Sp_{2n}(\Fqbar)$ where $2 \nmid q$. For $\lambda \in \fP_{2n}^C$, let $|\cO^{Sp}_{\lambda}(\Fq)|$ be the number of $\Fq$-points in the nilpotent orbit of $\Lie Sp_{2n}(\Fqbar)$ whose elements are of Jordan type $\lambda$ (when such an orbit is split over $\Fq$). Then  \cite[Lemma 5.1]{rs18} states that we have
\[\sum_{\lambda \in \fP_{2n}^C, \floor{l(\lambda)/2}=k}|\cO^{Sp}_{\lambda}(\Fq)|=q^{(n-k)(n-k-1)}(q^2-1)^{n-k}\frac{\dbs{n}!^2}{\dbs{n-k}!\dbs{k}!^2}.\]
(See also \cite{lus76}.) Therefore, it suffices to show that 
\[ \sum_{(\mu,\nu) \in \fP_{n,2}, l(\mu,\nu)=k} |\cO_{\mu,\nu}(\bF_q)| = \sum_{\lambda \in \fP_{2n}^C, \floor{l(\lambda)/2}=k}|\cO^{Sp}_{\lambda}(\Fq)|.\]

To this end, following \cite{ahs11}, first we define $\Phi^C: \fP_{n,2} \rightarrow \fP_{2n}$ as follows. For $(\mu, \nu) \in \fP_{n,2}$ such that $\mu=(\mu_1, \mu_2, \ldots)$ and $\nu=(\nu_1, \nu_2, \ldots)$, we consider the sequence $(2\mu_1, 2\nu_1, 2\mu_2, 2\nu_2, \ldots)$ and substitute any two consecutive integers $s, t$ such that $s<t$ with $(s+t)/2, (s+t)/2$, respectively. (These substitutions do not overlap with one another.) Then the result is a partition of $2n$ which is set to be the image of $(\mu,\nu)$ under $\Phi^C$. Note that we have $l(\mu,\nu) = \floor{l(\Phi^C(\mu,\nu))/2}$.

We define another map $-^C: \fP_{n,2} \rightarrow \fP_{n,2}$ as follows. For $(\mu,\nu)$ as above, whenever we have $\mu_i < \nu_i-1$ we replace $\mu_i, \nu_i$ with $\floor{(\mu_i+\nu_i)/2}, \ceil{(\mu_i+\nu_i)/2}$, respectively, and whenever we have $\nu_i<\mu_{i+1}+1$ we replace $\nu_i, \mu_{i+1}$ with $\floor{(\nu_i+\mu_{i+1})/2}, \ceil{(\nu_i+\mu_{i+1})/2}$, respectively. (These substitutions do not overlap with one another.) Then the result is again an element of $\fP_{n,2}$ denoted by $(\mu, \nu)^C$. Note that we have $l(\mu,\nu) = l((\mu,\nu)^C)$.

By \cite[Theorem 2.23(2)]{ahs11}, when $\Phi^C(\rho, \sigma)\in \fP_{2n}^C \subset \fP_{2n}$ we have
\[ \sum_{(\mu,\nu) \in \fP_{n,2},(\mu,\nu)^C=(\rho, \sigma)} |\cO_{\mu,\nu}(\bF_q)| = |\cO^{Sp}_{\Phi^C(\rho,\sigma)}(\Fq)|.\]
On the other hand, by \cite[Proposition 2.4]{ahs11}, $\Phi^C$ is a bijection between $\fP_{2n}^C$ and the image of $-^C$. Since $l(\mu,\nu) = l(\rho, \sigma) = \floor{l(\Phi^C(\rho, \sigma))/2}$, it follows that
\[ \sum_{(\mu,\nu) \in \fP_{n,2}, l(\mu,\nu)=k} |\cO_{\mu,\nu}(\bF_q)| = \sum_{\lambda \in \fP_{2n}^C, \floor{l(\lambda)/2}=k}|\cO^{Sp}_{\lambda}(\Fq)|\]
as desired.

\subsection{More on the positivity of $q$-Kreweras numbers} \label{sec:bcposmore}
In Theorem \ref{thm:bckre}, for odd $t$ we have $\Krew(W, \chi^{\mu,\nu},t;q) \in \bN[q]$ if and only if the first condition of Lemma \ref{lem:hm} is invalid for any $r \in (\mu+\nu)\cup(\mu_{\geq 2}\cup \nu)$. Here we give other equivalent interpretations of this condition and prove Theorem \ref{thm:pos} for type $BC_n$.

Recall the Springer correspondence of $\Lie Sp_{2n}$ in characteristic 2 discussed in \ref{sec:liesp}.
Following \cite[Section 4.2]{kim:liesp}, we define the notion of critical values. We say that $(\lambda, \kappa)\in \Omega$ is critical at $r \in \lambda$, or $r$ is a critical value of $(\lambda, \kappa)$, if
\begin{enumerate}[label=\textbullet]
\item $\kappa(r) \neq 0$,
\item for $r' \in \und{\lambda}$, if $r'<r$ then $\kappa(r')<\kappa(r)$, and 
\item for $r' \in \und{\lambda}$, if $r'>r$ then $r'-\kappa(r')<r-\kappa(r)$.
\end{enumerate}
Using this notion, we may decompose a nilpotent element into a ``distinguished part'' and an ``induced part''. Indeed, we define $\tilde{m}(r)$ to be
\begin{enumerate}[label=\textbullet]
\item 0 if $r$ is critical,
\item 1 if ($r$ is critical and) $m_\lambda(r)$ is odd, and
\item 2 if $r$ is critical and $m_\lambda(r)$ is even.
\end{enumerate}
(Indeed, it is easy to show that $r$ is critical if $m_\lambda(r)$ is odd.)
Now we set $(\tilde{\lambda}, \tilde{\kappa})$ to be $\tilde{\lambda} = (1^{\tilde{m}(1)}2^{\tilde{m}(2)}\cdots)$ and $\tilde{\kappa} =\kappa|_{\tilde{\lambda}}$. Also we set $(\lambda', \kappa')$ to be $(\lambda'=1^{m_\lambda(1)-\tilde{m}(1)}2^{m_\lambda(2)-\tilde{m}(2)}\cdots)$ (so that $\tilde{\lambda}+\lambda'=\lambda$) and $\kappa'=0$. If we choose $\tilde{N}\in \Lie Sp_{|\tilde{\lambda}|}(\Ftbar)$ and $N'\in \Lie Sp_{|\lambda'|}(\Ftbar)$ such that their orbits are parametrized by $(\tilde{\lambda}, \tilde{\kappa})$ and $(\lambda', \kappa')$, respectively, then it is not hard to show that $N \simeq \tilde{N}\oplus N'$, i.e. there exists a direct sum decomposition $\Ftbar^{2n}=\tilde{V}\oplus V'$ such that $N|_{\tilde{V}}\simeq \tilde{N}$ and $N|_{V'}\simeq N'$. Moreover, one can show that $\tilde{N}$ is a distinguished nilpotent element and $N'$ is a regular nilpotent element in a certain Levi subalgebra $\Lie Sp_{2n}(\Ftbar)$ all of whose simple factors are of type $A$.

\begin{thm} \label{thm:bcpos}Suppose that $\iota(\lambda, \kappa) = (\mu,\nu)$. Then the following conditions are equivalent:
\begin{enumerate}
\item $\Krew(W, \chi^{\mu,\nu}, t;q) \in \bN[q]$, i.e. $L(\mu,\nu)=0$.
\item There exists at most one critical value of $(\lambda, \kappa)$, and if $r$ is such a value then $m_\lambda(r)$ is odd.
\item The nilpotent orbit in $\Lie Sp_{2n} (\Ftbar)$ parametrized by $(\lambda, \kappa)$ contains a regular nilpotent element in some Levi subalgebra.
\item $\mu_1=\mu_2=\cdots=\mu_{l(\nu)}=\mu_{l(\nu)+1}.$ (This condition is still satisfied when $\mu=\emptyset$ or $\nu=\emptyset$.)
\end{enumerate}
\end{thm}
\begin{proof} (2) $\Leftrightarrow$ (3) is easily deduced from the argument right above the theorem. (This is equivalent to saying that $\tilde{N}$ is a regular nilpotent element of $\Lie Sp_{|\tilde{\lambda}|}(\Ftbar)$.) The equivalence of (2) and (4) is easily shown using the definition of $\iota$. 

We prove (1) $\Rightarrow$ (4). Let us set $\Lambda=(\mu_1, \nu_1, \mu_2, \nu_2, \ldots)=(a_1, a_2, a_3, a_4, \ldots)$ and for each $r \in (\mu+\nu)\cup(\mu_{\geq 2} +\nu)$ we attach the interval $I(r)=[s(r), e(r)]$ where $s(r)$ (resp. $e(r)$) is the smallest (resp. largest) index so that $a_{s(r)}+a_{s(r)+1} = r$ (resp. $a_{e(r)}+a_{e(r)+1} = r$). Then (1) is equivalent to that there is no $r$ such that $s(r)$ and $e(r)$ are both even. Since these intervals are pairwise disjoint and their union is $[1, l(\Lambda)]$, we have two possible cases:
\begin{enumerate}
\item If $l(\lambda)$ is even, i.e. $l(\mu)\leq l(\nu)$, then all the intervals are of the form $[$odd, even$]$.
\item If $l(\lambda)$ is odd, i.e. $l(\mu)>l(\nu)$, then there exists a unique $r' \in (\mu+\nu)\cup(\mu_{\geq 2} +\nu)$ so that if $r>r'$ then $I(r)$ is $[$odd, even$]$, $I(r')$ is $[$odd, odd$]$, and if $r<r'$ then $I(r)$ is $[$even, odd$]$. 
\end{enumerate}
In the first case, direct calculation shows that $a_1=a_3=\cdots = a_{l(\Lambda)+1}=0$, i.e. $\mu_1=\mu_2=\cdots=\mu_{l(\nu)}=\mu_{l(\nu)+1}=0$. In the second case, similarly we have $a_1=a_3=\cdots=a_{e(r')}$ and $a_{e(r')+1}=\cdots=a_{l(\Lambda)+1}=0$, i.e. $l(\nu)\leq (e(r')-1)/2$ and $\mu_1=\mu_2=\cdots=\mu_{(e(r')+1)/2}$. This proves (1) $\Rightarrow$ (4). Conversely, if $\mu_1=\mu_2=\cdots=\mu_{l(\nu)}=\mu_{l(\nu)+1}$ then it is easy to deduce (1) using the definition of $\iota$. It suffices for the proof.
\end{proof}
Suppose that $(\lambda, \kappa) \in \Omega$ satisfies the second condition of Theorem \ref{thm:bcpos}. Let $2b$ be the unique critical value of $(\lambda, \kappa)$. (If there is no critical value, then we set $2b=0$.) From the definition of a critical value, for any $r\in \lambda$ we have $\kappa(r) = b$ if $r\geq 2b$, $\kappa(r) = r-b$ if $b\leq r<2b$, and $\kappa(r)=0$ if $r<b$. Thus, if we let $\lambda=(a_1, a_1, a_2, a_2, \ldots, a_m, a_m, b, a_{m+1}, a_{m+1}, \ldots, a_k, a_k)$ where $a_1\geq \cdots \geq a_m>b\geq a_{m+1} \geq \cdots \geq a_k$ then by direct calculation we have
\[\iota(\lambda, \kappa) = ((b, b, \ldots, b, a_{m+1}, \ldots, a_k), (a_1-b, a_2-b, \ldots, a_m-b))\]
where $b$ is repeated $m+1$ times in the first factor. 
Note that this is equivalent to the description of $\Phi$ in Section \ref{sec:pos} for type $BC_n$, thus it proves Theorem \ref{thm:pos} in type $BC_n$.

\subsection{Cyclic sieving}
We prove Theorem \ref{thm:cyc} in our setting. Here we set $t=2ns+1$ for some $s \in \bN$ (note that the Coxeter number of $W$ is $2n$). As in \cite[6.1]{rs18}, for $d \in \bN$ we set $d^-=d/\gcd(d,2)$ and $d^+ = 2d^-$. Then for any $N \in \bN$ we have $d \mid 2N \Leftrightarrow d^+ \mid 2N \Leftrightarrow d^- \mid N.$ 

Recall that each parabolic subgroup $P \subset W$ may be identified with $H_{b}\times \sym_{a_1} \times \cdots \times \sym_{a_k}$ where $H_{b}$ is a Coxeter group of type $BC_{b}$ and $b+a_1+\cdots+a_k = n$.  (Here $b$ can be zero.) Without loss of generality we may assume that $a_1\geq \cdots\geq a_k$. Here we set $(\lambda, \kappa) \in \Omega$ as described at the end of \ref{sec:bcposmore} and let $(\mu,\nu) = \iota(\lambda, \kappa)$ so that $\Phi(P) = \chi^{(\mu,\nu)}$.

If we regard $W$ as the Weyl group of $Sp_{2n}(\bC)$ then $P$ corresponds to the nilpotent orbit of $\Lie  Sp_{2n}(\bC)$ whose elements are of Jordan type $\lambda$. Therefore, using the result of \cite[Section 6]{rs18}, here we only need to show that:
\begin{enumerate}
\item If $d \mid 2ns$ but $d \nmid 2\floor{m_\lambda(r)/2}$ for some $r$, then $\Krew(W, \Phi(P), 2ns+1;\omega_d)=0$.
\item If $d \mid 2ns$ and $d \mid 2\floor{m_\lambda(r)/2}$ for every $r$, then
\[\Krew(W, \Phi(P), 2ns+1;\omega_d) =\binom{ns/d^-}{ns/d^- - \floor{l(\lambda)/2}/d^-, \floor{m_\lambda(\bZ_{>0})/2}/d^-}.\]
(Here we use the fact that $m_\lambda(r)$ is odd for at most one $r\in \lambda$, so $\sum_{r\in \bZ_{>0}}\floor{m_\lambda(r)/2} =\floor{l(\lambda)/2}$.)
\end{enumerate}
Let us embark on the proof. From Theorem \ref{thm:bckre}, $\Krew(W, \Phi(P), 2ns+1;\omega_d) $ is equal to
\[\omega_d^{(2ns+1)(n-l(\mu,\nu))+l(\mu,\nu)^2-2z(\mu,\nu) +d(\mu,\nu)-n}
\qbinom{ns}{ns-l(\mu,\nu),m_{\mu,\nu}(\bZ_{>0})}_{\omega_d^2}.\]
Here, $\Phi(P)=(\mu,\nu) = ((b, b, \ldots, b, a_{m+1}, \ldots, a_k), (a_1-b, a_2-b, \ldots, a_m-b))$ where $m\in [0,k]$ is the largest index so that $a_m>b$ and $b$ is repeated $m+1$ times in $\mu$. Also, note that in this case $L(\mu,\nu)=0$.
Now we observe the following lemma.
\begin{lem} Let $\lambda, \mu,\nu$ be as above. Then we have $m_{\mu,\nu}(b) = (m_\lambda(b)-1)/2$ and $m_{\mu,\nu}(r) = m_\lambda(r)/2$ if $r \neq b$. In any case, we have $m_{\mu,\nu}(r) = 
\floor{m_{\lambda}(r)/2}$. Also, it follows that $l(\mu,\nu) = \sum_{r\geq 1} m_{\mu,\nu}(r)= \floor{l(\lambda)/2}$.
\end{lem}
\begin{proof} It is straightforward from the description of $(\mu,\nu)$.
\end{proof}
\begin{rmk} In general if $\iota(\lambda, \chi) = (\mu,\nu)$ (see \ref{sec:liesp} for the definition of $\iota$) then we have:
\begin{enumerate}[label=\textbullet]
\item if $r$ is not a critical value, then $m_{\mu,\nu}(r) = m_\lambda(r)/2$.
\item if $m_\lambda(r)$ is odd, then ($r$ is critical and) $m_{\mu,\nu}(r) = (m_\lambda(r)-1)/2$.
\item if $m_\lambda(r)$ is even and $r$ is a critical value, then $m_{\mu,\nu}(r) = (m_\lambda(r)-2)/2$.
\end{enumerate}
\end{rmk}
We also recall a useful lemma when calculating the specialization of $q$-multinomials.
\begin{lem}[{\cite[Lemma 6.3]{rs18}}] \label{lem:calqbin} Suppose that $d \mid 2N$. Then for a sequence of nonnegative integers $a_1, \ldots, a_l$ so that $N=a_1+\cdots+a_l$, we have
\[\qbinom{N}{a_1, \ldots, a_l}_{\omega_d^2} = \left\{
\begin{aligned}
&\binom{N/d^-}{a_1/d^-, \ldots, a_l/d^-} &\textnormal{ if } d^- \mid a_i \textnormal{ for each } i,
\\&0 & \textnormal{ otherwise}.
\end{aligned}
\right.
\]
\end{lem}

Now suppose that $d \mid 2ns$ but $d \nmid 2\floor{m_\lambda(r)/2}=2m_{\mu,\nu}(r)$ for some $r$. Then it is clear that \[\qbinom{ns}{ns-l(\mu,\nu),m_{\mu,\nu}(\bZ_{>0})}_{\omega_d^2}=0, \]
thus $\Krew(W, \Phi(P), 2ns+1;\omega_d)=0$ as required. From now on we assume that $d \mid 2ns$ and $d \mid 2\floor{m_\lambda(r)/2}=2m_{\mu,\nu}(r)$ for every $r\in \bZ_{>0}$. (It also follows that $d\mid2\floor{l(\lambda)/2} =2l(\mu,\nu)$.) Then by Lemma \ref{lem:calqbin}, $\Krew(W, \Phi(P), 2ns+1;\omega_d)$ equals
\[\omega_d^{(2ns+1)(n-l(\mu,\nu))+l(\mu,\nu)^2-2z(\mu,\nu) +d(\mu,\nu)-n}
\binom{ns/d^-}{ns/d^--\floor{l(\lambda)/2}/d^-,\floor{m_{\lambda}(\bZ_{>0})/2}/d^-}.\]
Thus it remains to show that the power of $\omega_d$ in the above expression is 1.

In our case, it is easy to see that each of $\mu^T_i-1$ and $\nu^T_i$ are the sum of $m_{\mu,\nu}(r)$ for some $r$, and thus $z(\mu,\nu) = 2z(\mu)+2z(\nu)+|\nu|$ is easily seen to be a multiple of $d^-$. (Recall that $2z(\mu) = \sum_{i\geq 1}\mu_i^T(\mu_i^T-1)$ and $2z(\nu) = \sum_{i\geq 1}\nu_i^T(\nu_i^T-1)$.) Thus we have
\begin{align*}
(2ns+1)(n-l(\mu,\nu))+l(\mu,\nu)^2-2z(\mu,\nu) +d(\mu,\nu)-n \equiv 0 \pmod {d^-},
\end{align*}
since $d^- \mid d(\mu,\nu)= \sum_{r >0} m_{\mu,\nu}(r)(m_{\mu,\nu}(r)+1)$. 

Therefore, if $d$ is odd then $d=d^-$ and it suffices for the proof. It remains to assume that $d$ is even and $d=2d^-$. Here we have
\begin{align*}
\\&(2ns+1)(n-l(\mu,\nu))+l(\mu,\nu)^2-2z(\mu,\nu) +d(\mu,\nu)-n 
\\&\equiv l(\mu,\nu)^2-l(\mu,\nu)+d(\mu,\nu) 
\\&\equiv \left(\sum_{r\geq 1}m_{\mu,\nu}(r)\right)^2+\sum_{r\geq 1}m_{\mu,\nu}(r)^2
\\&\equiv 2\sum_{r\geq 1}m_{\mu,\nu}(r)^2+2\sum_{r>s \geq1}m_{\mu,\nu}(r)m_{\mu,\nu}(s) \equiv 0  \pmod d.
\end{align*}
Thus the claim also holds when $d$ is even.

\section{Type $H_3$}\label{sec:h3}
In this section we assume that $W$ is of type $H_3$. 
\subsection{$q$-Kreweras numbers and positivity.} Here we list $\Krew(W, \chi, t;q)$ for $\chi \in \Irr(W)$. (We acknowledge that these numbers were also calculated by Eric Sommers.)
\begin{center}
\begin{tabular}{| >{$}c<{$} | >{$}c<{$} | >{$}c<{$} | >{$}c<{$}|}
\hline
\Irr(W) & \textnormal{$q$-Kreweras}&\br{Q_{\chi}, V}& parabolic
\\\hline
\phi_{1,15}&[t-1][t-5][t-9]/([2][6][10])&q^9 + q^5 + q&triv
\\\phi_{1,0}&q^{3t-3}&0&H_3
\\\phi_{5,5}&q^{2t-10}[t-1]/[2]&q&A_1\times A_1
\\\phi_{5,2}&q^{2t-6}(q-1)[t-1]&q&
\\\phi_{3,6}&q^{t-7}(q-1)[t-1][t-5]/[2]&q^5 + q&
\\\phi_{3,8}&q^{t-9}[t-1][t-5]/[2]^2&q^5 + q&A_1
\\\phi_{3,1}& q^{2t-4}(q-1)[t-1]&q&
\\\phi_{3,3}& q^{2t-6}[t-1]/[2]&q&I_{2}(5)
\\\phi_{4,3}& q^{2t-8}(q-1)[t-1]&q&
\\\phi_{4,4}&q^{2t-8}[t-1]/[2]&q&A_2
\\\hline
\end{tabular}
\end{center}
The first column represents the irreducible representations of $W$. Here we use Carter's notation \cite[13.2]{car93}. The second column represents corresponding $q$-Kreweras numbers. The third column represents $\br{Q_\chi, V}$ where $V$ is the reflection representation of $W$. The fourth column represents the type of each parabolic subgroup $P \subset W$ where $\Phi(P) = \chi$. Here we do not distinguish different parabolic subgroups of type $A_1$ as they all correspond to the same representation $\phi_{3,8}$. 

It is clear from the formula that if $t \equiv 1,5,9 \pmod{10}$ we have $\Krew(W, \chi, t;q) \in \bN[q]$ if and only if $\chi \in \im \Phi$.

\subsection{Relation with $q$-Narayana numbers} We have
\[\Nar(W,k,t;q) = (q^{t-4k-1})^{3-k} \qbinom{3}{k}_{q^4} \frac{(q^{t-1};q^{-4})_k}{(q^{2};q^4)_k},\]
or more precisely
\begin{align*}
\Nar(W,0,t;q) &= q^{3(t-1)},
\\\Nar(W,1,t;q) &= q^{2(t-5)} (q^8+q^4+1) \frac{[t-1]}{[2]},
\\\Nar(W,2,t;q) &= q^{t-9}  (q^8+q^4+1) \frac{[t-1][t-5]}{[2][6]},
\\\Nar(W,3,t;q) &=  \frac{[t-1][t-5][t-9]}{[2][6][10]}.
\end{align*}
Thus Theorem \ref{thm:nar} holds in this case by direct calculation.

\subsection{Cyclic sieving}
We identify $W$ with $\br{s_1, s_2, s_3 \mid s_1^2, s_2^2, s_3^3, (s_1s_2)^5, (s_2s_3)^3, (s_1s_3)^2}$ and fix a (standard) Coxeter element $c=s_1s_2s_3$. Let us label all the reflections in $W$ as
\begin{align*}
&r_1=s_1,
&&r_2=s_2,
&&r_3=s_{1232321},
&&r_4=s_{232},
&&r_5=s_{31231},
\\&r_6=s_{121},
&&r_7=s_{123123121},
&&r_8=s_{2312312312312},
&&r_9=s_{323123123},
&&r_{10}=s_{323},
\\&r_{11}=s_3,
&&r_{12}=s_{12321},
&&r_{13}=s_{2312312},
&&r_{14}=s_{31231231231},
&&r_{15}=s_{23232},
\end{align*}
and all the products of two reflections in $W$ as
\begin{align*}
&o_1=s_{12},
&&o_2=s_{12312321},
&&o_3=s_{1232312312},
&&o_4=s_{23231231},
&&o_5=s_{3123},
\\&o_6=s_{23},
&&o_7=s_{1231},
&&o_8=s_{1232},
&&o_9=s_{231231},
&&o_{10}=s_{123232},
\\&o_{11}=s_{31},
&&o_{12}=s_{123121},
&&o_{13}=s_{123123123121},
&&o_{14}=s_{323123123123},
&&o_{15}=s_{323123}.
\end{align*}
(Here, $s_{abc\cdots} = s_as_bs_c \cdots$.)
Then we have $\NC^{(1)}(W) = \{id, c\} \cup \{r_i\mid r \in [1,15]\} \cup \{o_i \mid i \in [1,15]\}$. Moreover, direct calculation shows that $cr_{i}c^{-1} = r_{i+1}$ and $co_{i}c^{-1} = o_{i+1}$ for $i \not\in \{5, 10, 15\}$, and  $cr_{i}c^{-1} = r_{i-4}$ and $co_{i}c^{-1} = o_{i-4}$ for $i \in \{5, 10, 15\}$.

From now on we calculate the fixed points of $\NC^{(s)}(W)$ under the $\bZ/10s$-action defined in \ref{sec:cycsi} case-by-case. To this end, recall the definition of $\delta$-sequence \[(\delta_1, \ldots, \delta_{s-1}, \delta_s) = (w_1^{-1}w_2, \ldots, w_{s-1}^{-1}w_s, w_{s}^{-1}c)\]
for each $(w_1, w_2, \ldots, w_{s}) \in \NC^{(s)}$. It is clear that $\delta_i \in \NC^{(1)}(W)$ for each $i$. In terms of $\delta$-sequence the cyclic action is given by
\[(\delta_1, \ldots, \delta_{s-1}, \delta_s)  \mapsto (c\delta_s c^{-1}, \delta_1, \ldots, \delta_{s-1}). \]
From now on let us write $\varphi(w_1,s,d)$ to denote the number of fixed points in $\NC^{(s)}(W)$ starting with $w_1 \in W$ by an order $d$ element (where $d\mid 10s$). Also let $\Delta$ be the sequence obtained from $(\delta_1, \ldots, \delta_{s-1}, \delta_s)$ by removing the identity elements.
\begin{enumerate}
\item $w_1= c$. In this case $\Delta = \emptyset$. Thus $\varphi(c, s,d) = 1$.
\item $w_1=o_i$ for some $i\in [1,15]$. In this case $\Delta= (o_i^{-1}c)$ where $o_i^{-1}c$ is a reflection. Thus in this case $\varphi(o_i,s,d) = s$ if and only if $d\mid 2$ and otherwise $\varphi(o_i,s,d) =0$.
\item $w_1= r_i$, $i \in [1,5]$. By conjugation by $c$ it suffices to consider the case when $r_i=s_1$. Then possible $\Delta$ are:
\begin{align*}
(r_2, r_{11}),\quad (r_4,r_2),\quad (r_{15}, r_4),\quad (r_{10}, r_{15}), \quad(r_{11}, r_{10}),\quad (o_{6}).
\end{align*}
Thus direct calculation shows that $\varphi(r_i, s, d) = 5s(s-1)/2+s$ if and only if $d \mid 2$ and otherwise $\varphi(r_i, s, d) = 0$.
\item $w_1= r_i$, $i \in [6,10]$. By conjugation by $c$ it suffices to consider the case that $r_i=s_{121}$. Then possible $\Delta$ are:
\begin{align*}
(r_1, r_{11}),\quad (r_{11},r_1), \quad (o_{11}).
\end{align*}
Thus direct calculation shows that $\varphi(r_i, s, d) = s^2$ if and only if $d \mid 2$ and otherwise $\varphi(r_i, s, d) = 0$.
\item $w_1= r_i$, $i \in [11,15]$. By conjugation by $c$ it suffices to consider the case that $r_i=s_{3}$. Then possible $\Delta$ are:
\begin{align*}
(r_1, r_{10}),\quad (r_5, r_{1}), \quad (r_{10},r_5), \quad (o_{5}).
\end{align*}
Thus direct calculation shows that $\varphi(r_i, s, d) = 3s(s-1)/2+s$ if and only if $d \mid 2$ and otherwise $\varphi(r_i, s, d) = 0$.
\item $w_1=id$. In this case possible $\Delta$ can be read from above calculations. Let us denote by $\varphi_{\Delta}(id, s, d)$ the number of fixed elements with fixed $\Delta$.
\begin{enumerate}
\item $\Delta=(c)$. Direct calculation shows that $\varphi_\Delta(id, s, d) = s$ if and only if $d \mid 10$ and otherwise $\varphi_\Delta(id, s, d) = 0$.
\item $\Delta=(o_i, o_i^{-1}s)$ or $\Delta=(o_i^{-1}s, o_i)$ for some $i \in [1,15]$. It is clear that $\varphi_\Delta(id, s, d) = s(s-1)/2$ if and only if $d \mid 2$ and otherwise $\varphi_\Delta(id, s, d) = 0$.
\item $\Delta = (r_i, r_j, r_k)$ for some $i,j,k \in [1,15]$. Direct calculation shows that there are 50 possible $\Delta$, and we have $\varphi_{\Delta}(id, s, d) = s(s-1)(s-2)/6$ if $d\mid 2$ and $\varphi_{\Delta}(id, s, d)=0$ if $d\neq 1,2,3, 6$. Moreover, if $3\mid s$ and $\Delta$ is one of
\[(r_{1}, r_{4}, r_{2}), (r_{2}, r_{5}, r_{3}),  (r_{3}, r_{1}, r_{4}),(r_{4}, r_{2}, r_{5}), (r_{5}, r_{3}, r_{1})\]
then the $\delta$-sequence $(\delta_1, \ldots, \delta_s)$ where $\Delta=(\delta_i, \delta_{s/3+i}, \delta_{2s/3+i})$ for some $i \in [1, s/3]$ is fixed by an element of order 3 and 6. Thus we have $\varphi_{\Delta}(id, s, 6)=\varphi_{\Delta}(id, s, 3)= s/3$. Otherwise we have $\varphi_{\Delta}(id, s, 3)= 0$.
\end{enumerate}
In sum, we have:
\[
\varphi(id, s, d) = \left\{
\begin{aligned}
&s(5s-2)(5s-4)/3 & \textnormal{ if } d\in\{1,2\},
\\&5s/3 & \textnormal{ if } 3\mid s \textnormal{ and } d \in \{3,6\},
\\&s& \textnormal{ if } d \in \{5, 10\},
\\&0 & \textnormal{ otherwise.}
\end{aligned}
\right.
\]
\end{enumerate}

Now we consider the $q$-Kreweras numbers $\Krew(W, \Phi(P), 10s+1;\omega_d)$ where $d \mid 10s$. Direct calculation shows that:
\begin{enumerate}[label=\textbullet]
\item $P$ is of type $H_3$. Then $\Phi(P) = \phi_{1,0}$. For $w\in \NC^{(1)}(W)$, $V^w \in W\cdot V^P$ if and only if $w=c$. We have $\Krew(W, \phi_{1,0},10s+1;q) = q^{30s}$, thus $\Krew(W, \phi_{1,0},10s+1;\omega_d)=1$ whenever $d \mid 10s$. This coincides with $\varphi(c, s, d)=1$. 
\item $P$ is of type $I_2(5)$. Then $\Phi(P) =\phi_{3,3}$. For $w\in \NC^{(1)}(W)$, $V^w \in W\cdot V^P$ if and only if $w=o_{i}$ for $i \in [1,5]$. We have $\Krew(W,\phi_{3,3} ,10s+1;q) = q^{20s-4}[10s]/[2]$, thus $\Krew(W,\phi_{3,3} ,10s+1;\omega_d)$ equals $5s$ if $d \mid 2$ and $0$ otherwise. Thus coincides with $\sum_{i=1}^5\varphi(o_i,s,d)$.
\item $P$ is of type $A_2$. Then $\Phi(P) =\phi_{4,4}$. For $w\in \NC^{(1)}(W)$, $V^w \in W\cdot V^P$ if and only if $w=o_{i}$ for $i \in [6,10]$. We have $\Krew(W,\phi_{4,4} ,10s+1;q) = q^{20s-6}[10s]/[2]$, thus $\Krew(W,\phi_{4,4},10s+1;\omega_d)$ equals $5s$ if $d \mid 2$ and $0$ otherwise. Thus coincides with $\sum_{i=6}^{10}\varphi(o_i,s,d)$.
\item $P$ is of type $A_1\times A_1$. Then $\Phi(P) =\phi_{5,5}$. For $w\in \NC^{(1)}(W)$, $V^w \in W\cdot V^P$ if and only if $w=o_{i}$ for $i \in [11,15]$. We have $\Krew(W,\phi_{5,5} ,10s+1;q) = q^{20s-8}[10s]/[2]$, thus $\Krew(W,\phi_{5,5},10s+1;\omega_d)$ equals $5s$ if $d \mid 2$ and $0$ otherwise. Thus coincides with $\sum_{i=11}^{15}\varphi(o_i,s,d)$.
\item $P$ is of type $A_1$. Then $\Phi(P) =\phi_{3,8}$. For $w\in \NC^{(1)}(W)$, $V^w \in W\cdot V^P$ if and only if $w=r_{i}$ for $i \in [1,15]$. We have $\Krew(W,\phi_{3,8} ,10s+1;q) = q^{10s-8}[10s][10s-4]/[2]^2$, thus $\Krew(W,\phi_{3,8},10s+1;\omega_d)$ equals $5s(5s-2)$ if $d \mid 2$ and $0$ otherwise. Thus coincides with $\sum_{i=1}^{15}\varphi(r_i,s,d)$.
\item $P=\{id\}$. Then $\Phi(P) =\phi_{1,15}$. For $w\in \NC^{(1)}(W)$, $V^w \in W\cdot V^P$ if and only if $w=id$. We have $\Krew(W,\phi_{1,15} ,10s+1;q) = [10s][10s-4][10s-8]/([2][6][10])$, thus 
\[\Krew(W,\phi_{1,15} ,10s+1;\omega_d)=\left\{
\begin{aligned}
&s(5s-2)(5s-4)/3& \textnormal{ if } d \in \{1,2\},
\\&5s/3 & \textnormal{ if } 3 \mid s \textnormal{ and } d \in \{3,6\},
\\&s & \textnormal{ if } d \in \{5,10\},
\\&0 & \textnormal{ otherwise.}
\end{aligned}
\right.
\]
This coincides with $\varphi(id,s,d)$.
\end{enumerate}
We exhaust all the possible cases and thus Theorem \ref{thm:cyc} is valid when $W$ is of type $H_3$.

%
%
%
%
%
%
%
%

%
%
%

\section{Type $I_2(m)$}\label{sec:i2m}
In this section we assume that $W$ is of type $I_2(m)$ for $m \geq 2$.
\subsection{$q$-Kreweras numbers and positivity.} Here we list $\Krew(W, \chi, t;q)$ for $\chi \in \Irr(W)$ in the same way as type $H_3$ case. When $m$ is even we have:
\begin{center}
\begin{tabular}{| >{$}c<{$} | >{$}c<{$} | >{$}c<{$} | >{$}c<{$}|}
\hline
\Irr(W) & \textnormal{$q$-Kreweras}&\br{Q_{\chi}, V}& parabolic
\\\hline
\phi_{1,0}&q^{2t-2}&0&I_{2}(m)
\\\phi_{1,m/2}'&q^{t-m+1}[t-1]/[2]&q&A_1'
\\\phi_{1,m/2}''&q^{t-m+1}[t-1]/[2]&q&A_1''
\\\phi_{1,m}&[t-1][t-m+1]/([2][m])&q^{m-1}+q&triv
\\\phi_{2,r}, r\in [1,m/2-1]&q^{t-2r-1}(q^{t-1}-1)&q&
\\\hline
\end{tabular}
\end{center}
When $m$ is odd we have:
\begin{center}
\begin{tabular}{| >{$}c<{$} | >{$}c<{$} | >{$}c<{$} | >{$}c<{$}|}
\hline
\Irr(W) & \textnormal{$q$-Kreweras}&\br{Q_{\chi}, V}& parabolic
\\\hline
\phi_{1,0}&q^{2t-2}&0&I_{2}(m)
\\\phi_{1,m}&[t-1][t-m+1]/([2][m])&q^{m-1}+q&triv
\\\phi_{2,r}, r\in [1,(m-3)/2]&q^{t-2r-1}(q^{t-1}-1)&q&
\\\phi_{2,(m-1)/2}&q^{t-m+1}[t-1]&q&A_1
\\\hline
\end{tabular}
\end{center}
It is clear from the formula that if $t \equiv \pm1 \pmod{m}$ we have $\Krew(W, \chi, t;q) \in \bN[q]$ if and only if $\chi \in \im \Phi$.
\subsection{Relation with $q$-Narayana numbers} We have
\[\Nar(W, k,t;q) = q^{(t-(m-2)k-1)(2-k)} \qbinom{2}{k}_{q^{m-2}} \frac{(q^{t-1};q^{-m+2})_k}{(q^{2};q^{m-2})_k}\]
or more precisely
\begin{align*}
\Nar(W, k,0;q)  &= q^{2(t-1)},
\\\Nar(W, k,1;q)  &= q^{t-m+1}(q^{m-2}+1) \frac{[t-1]}{[2]},
\\\Nar(W, k,2;q) &=  \frac{[t-1][t-m+1]}{[2][m]}.
\end{align*}
Thus Theorem \ref{thm:nar} holds in this case by direct calculation.

\subsection{Cyclic sieving}
We identify $W$ with $\br{s_1, s_2 \mid s_1^2, s_2^2,  (s_1s_2)^m}$ and fix a (standard) Coxeter element $c=s_1s_2$. We have $\Rx(W)= \{c^{i-1}s_1c^{1-i} \mid 1\leq i \leq (m+1)/2\} \cup \{ c^{i-1}s_2c^{1-i}\mid 1\leq i \leq m/2]\}$. One can show that each element is fixed by conjugation by $c^k$ if and only if $m \mid k$ (resp. $m/2 \mid k$) if $m$ is odd (resp. even).

From now on we calculate the fixed points of $\NC^{(s)}(W)$ under the $\bZ/ms$-action defined in \ref{sec:cycsi} case-by-case in the same manner as type $H_3$ case. Let us write $\varphi(w_1,s,d)$ to denote the number of fixed points in $\NC^{(s)}(W)$ starting with $w_1 \in W$ by an order $d$ element (where $d\mid ms$). Also let $\Delta$ be the sequence obtained from $(\delta_1, \ldots, \delta_{s-1}, \delta_s)$ by removing the identity elements.

\begin{enumerate}
\item $w_1=c$. In this case $\Delta=\emptyset$. Thus $\phi(c, s, d) = 1$.
\item $w_1$ is a reflection. In this case $\Delta=(w_1c)$. Direct calculation shows that $\phi(w_1,s,d)$ equals $s$ if $d \mid \gcd(2,m)$ and 0 otherwise.
\item $w_1=id$. Let us denote by $\varphi_{\Delta}(id, s, d)$ the number of fixed elements with fixed $\Delta$. 
\begin{enumerate}[label=\textbullet]
\item If $\Delta=(c)$, direct calculation shows that $\phi_{(c)}(w_1,s,d)$ equals $s$ if $d \mid m$ and 0 otherwise. 
\item If $\Delta$ consists of two reflections, then there are $m$ possibilities of $\Delta$.
\begin{enumerate}[label=--] 
\item If $m$ is even then for any $\Delta = (r, r')$, $r$ and $r'$ are not conjugate by powers of $c$. Thus $\varphi_{\Delta}(w_1, s, d)$ equals $s(s-1)/2$ if $d \mid 2$ and 0 otherwise. 
\item If $m$ is odd then for any $\Delta = (r, r')$ we have $r'=c^{(m-1)/2}rc^{-(m-1)/2}$. Thus we have $\varphi_{\Delta}(w_1, s, d)$ equals $s(s-1)/2$ if $d=1$ and 0 if $d \nmid 2$. Moreover, if $2\mid s$ and $\Delta=(\delta_i, \delta_{s/2+i})$ for some $i \in [1, s/2]$ then the corresponding $\delta$-sequence is fixed by an order 2 element. Therefore $\varphi_{\Delta}(w_1, s, 2)$ equals $s/2$ if $2\mid s$ and 0 otherwise.
\end{enumerate}
\end{enumerate}
In sum, we have
\[\varphi(w_1, s, d)=\left\{ 
\begin{aligned}
&s(ms-m+2)/2 & \textnormal{ if } d=1 \textnormal{ or } [2 \mid m \textnormal{ and } d=2],
\\&ms/2& \textnormal{ if } 2 \nmid m, 2\mid s, \textnormal{ and } d=2,
\\&s & \textnormal{ if } d \mid m \textnormal{ and } d \not\in \{1, 2\}, \textnormal{ and}
\\&0 & \textnormal{ otherwise.}
\end{aligned}
\right.
\]
\end{enumerate}

Now we consider the $q$-Kreweras numbers $\Krew(W, \Phi(P), ms+1;\omega_d)$ where $d \mid ms$. Direct calculation shows that:
\begin{enumerate}[label=\textbullet]
\item $P$ is of type $I_2(m)$. Then $\Phi(P) = \phi_{1,0}$. For $w\in \NC^{(1)}(W)$, $V^w \in W\cdot V^P$ if and only if $w=c$. We have $\Krew(W, \phi_{1,0},ms+1;q) = q^{2ms}$, thus $\Krew(W, \phi_{1,0},ms+1;\omega_d)=1$ whenever $d \mid ms$. This coincides with $\varphi(c, s, d)=1$. 
\item $P$ is of type $A_1$, $m$ is odd. Then $\Phi(P) = \phi_{2, (m-1)/2}$. For $w\in \NC^{(1)}(W)$, $V^w \in W\cdot V^P$ if and only if $w$ is a reflection. We have $\Krew(W, \phi_{2,(m-1)/2},ms+1;q) = q^{ms-m+2}[ms]$, thus $\Krew(W, \phi_{2,(m-1)/2},ms+1;\omega_d)$ equals $ms$ if $d=1$ and 0 otherwise. This coincides with $\sum_{r\in \Rx(W)}\varphi(r, s, d)$.
\item $P$ is of type $A_1$, $m$ is even. By symmetry it suffices to assume that $P = \{id, s_1\}$ and $\Phi(P) = \phi_{1, m/2}'$. For $w\in \NC^{(1)}(W)$, $V^w \in W\cdot V^P$ if and only if $w\in \{c^js_1c^{-j} \mid j \in \bZ\}$ (there are $m/2$ of them). We have $\Krew(W, \phi_{1,m/2}',ms+1;q) = q^{ms-m+2}[ms]/[2]$, thus $\Krew(W, \phi_{1,m/2}',ms+1;\omega_d)$ equals $ms/2$ if $d\mid 2$ and 0 otherwise. This coincides with $\sum_{j=1 }^{m/2}\varphi(c^js_1c^{-j}, s, d)$.
\item $P=\{id\}$. Then $\Phi(P) = \phi_{1, m}$.  For $w\in \NC^{(1)}(W)$, $V^w \in W\cdot V^P$ if and only if $w=id$. We have $\Krew(W, \phi_{1,m}, ms+1;q) = [ms][ms-m+2]/([2][m])$, thus
\[\Krew(W, \phi_{1,m}, ms+1;\omega_d)=\left\{ 
\begin{aligned}
&s(ms-m+2)/2 & \textnormal{ if } d=1 \textnormal{ or } [2 \mid m \textnormal{ and } d=2],
\\&ms/2& \textnormal{ if } 2 \nmid m, 2\mid s, \textnormal{ and } d=2,
\\&s & \textnormal{ if } d \mid m \textnormal{ and } d \not\in \{1, 2\}, \textnormal{ and}
\\&0 & \textnormal{ otherwise.}
\end{aligned}
\right.
\]
This coincides with $\varphi(id,s,d)$.
\end{enumerate}
We exhaust all the possible cases and thus Theorem \ref{thm:cyc} is valid when $W$ is of type $I_2(m)$.

\bibliographystyle{amsalphacopy}
\bibliography{q_kre}

\end{document}